\documentclass{amsart}
\usepackage{cases}
\usepackage{graphicx}
\usepackage{amssymb}
\usepackage{amsfonts}
\usepackage{amsmath}
\usepackage{amscd}
\usepackage{xypic}
\usepackage{xcolor}
\newtheorem{thm}{Theorem}[section]
 \newtheorem{cor}[thm]{Corollary}
 \newtheorem{lem}[thm]{Lemma}
 \newtheorem{prop}[thm]{Proposition}
 \theoremstyle{definition}
 \newtheorem{defn}{Definition}[subsection]
 
 \theoremstyle{remark}
 \newtheorem{rem}{Remark}[section]
 \newtheorem{example}{Example}
 
 \numberwithin{equation}{section}
 
 \DeclareMathOperator{\IM}{Im}

\begin{document}

\title[Orbit configuration spaces ]
 {Orbit configuration spaces of small covers and quasi-toric manifolds }

\author{Junda Chen, Zhi L\"u and Jie Wu}

\address{School of Mathematical Sciences, Fudan University, Shanghai, 200433,  China}
\email{072018012@fudan.edu.cn}
\address{School of Mathematical Sciences, Fudan University\\ Shanghai\\ 200433\\ China}
\email{zlu@fudan.edu.cn}
\address{Department of Mathematics\\
National University of Singapore\\
Singapore 119260\\
Republic of Singapore}\email{matwujie@math.nus.edu.sg}
\urladdr{http://www.math.nus.edu.sg/\~{}matwujie}
\thanks{Research of the first author and the second author is supported by grants from NSFC (No. 10931005), Shanghai NSF (No. 10ZR1403600) and RFDP (No. 20100071110001).
Research of the third author is supported in part by the AcRF Tier 1 (WBS No. R-146-000-137-112) and Tier 2 (WBS No. R-146-000-143-112) of MOE of Singapore and a grant (No. 11028104) of NSFC}
\email{}


\begin{abstract}
In this article, we investigate the orbit configuration spaces of some equivariant closed manifolds over simple convex polytopes in toric topology, such as small covers,  quasi-toric manifolds and (real) moment-angle manifolds; especially for the cases of small covers and  quasi-toric manifolds.
These kinds of orbit configuration spaces are all non-free and noncompact, but still built via simple convex polytopes.
 We obtain an explicit formula of Euler characteristic for orbit configuration spaces of small covers and quasi-toric manifolds in terms of the $h$-vector of a simple convex polytope. As a by-product of our method, we also obtain a formula of Euler characteristic for the classical
configuration space, which generalizes the F\'elix-Thomas formula.
In addition, we also study the homotopy type of such orbit
configuration spaces. In particular, we determine an equivariant
strong deformation retract of the orbit configuration space of 2
distinct orbit-points in a small cover or a quasi-toric manifold,
which turns out that we are able to further study the algebraic
topology of such an orbit configuration space by using the
Mayer-Vietoris spectral sequence.
\end{abstract}

\maketitle

\section{Introduction}

\subsection{Notion and motivation} Let $G$ be a topological group and let $M$ be a $G$-space.
The \textit{$($ordered$)$ orbit configuration space} $F_G(M,k)$ is
defined by
$$
F_G(M,k)=\{(x_1,\ldots,x_k)\in M^{\times k} \ | \ G(x_i) \cap G(x_j)=\emptyset \textrm{ for } i\not=j\}
$$
with subspace topology, where $k\geq 2$ and $G(x)$ denotes the orbit at $x$. In the case where  $G$
acts trivially on $M$, the space $F_G(M,k)$ is the classical
configuration space denoted by $F(M,k)$.

The notion of configuration space had been introduced in physics in 1940s~\cite{n,vt} concerning the topology of configurations with various study on this important object since then. In mathematics, configuration spaces were first introduced by Fadell and
Neuwirth~\cite{fn} in 1962 with various applications~\cite{ar,bcww,bir,bo,co,t,v}. Since 1990s, the notion of configurations was introduced in robotics community to study safe-control problems of robots.
A new-created field in mathematics named \textit{topological robotics} was recently established by Ghrist and Farber~\cite{fa,Gh}, where the topology of configuration spaces on graphs plays an important role.
The orbit configuration spaces with labels provide combinatorial models for equivariant loop
spaces~\cite{x}. Moreover orbit configuration space is an analogy to fiber-type
arrangements~\cite{dco}. The fundamental groups of orbit configuration spaces enrich the theory of braids~\cite{ckx,cx}.

If the group $G$ acts properly discontinuously on a manifold $M$, there are various fibrations available related to $F_G(M,k)$~\cite{x}. Thus the standard methods of spectral sequences in algebraic topology can be used for studying the cohomology of $F_G(M,k)$. In particular, the cohomology of $F_{\mathbb{Z}_2}(S^n,k)$ has been determined in~\cite{fz,x2}, where $\mathbb Z_2$ acts (freely) on $S^n$ by antipodal map. In the case that $G$ does not freely on a manifold $M$,  the determination of the homotopy type or cohomology of $F_G(M,k)$ becomes much harder  because algebraic topology lacks sufficiently effective tools. For instance, the classical Fadell-Neuwirth fibration~\cite{fn} fails in non-free cases in general.

In 1991, Davis and Januszkiewicz \cite{dj} introduced four classes
of particularly nicely behaved manifolds over simple convex
polytopes--small covers, quasi-toric manifolds  and (real)
moment-angle manifolds, which have become  important  objects in
toric topology. Note that in their paper~\cite{dj}, Davis and
Januszkiewicz used the terminology ``toric manifold" rather than
``quasi-toric manifold", but the former has been used in algebraic
geometry as the meaning of smooth nonsingular toric variety, so
Buchstaber and Panov \cite{bp} began with the use of the terminology
``quasi-toric manifold". In addition,  (real) moment-angle manifolds
were named by Buchstaber and Panov \cite{bp} later  when they
studied the topology of (real) moment-angle manifolds as
submanifolds in polydisks. A {\em quasi-toric manifold} (resp. {\em
small cover}), as  the topological version of a compact non-singular
toric variety (resp. real toric variety),  is a smooth closed
manifold $M$ of dimension $2n$ (resp. dimension $n$) with a locally
standard action of torus $T^n$ (resp. real torus ${\Bbb Z}_2^n$)
such that its orbit space is a simple convex $n$-polytope $P$. A
{\em $($real$)$ moment-angle manifold} can directly be constructed
from a simple convex polytope $P$ such that it admits an action of
real torus or torus with $P$ as its orbit space. There are strong
links between topology and geometry of these equivariant manifolds
and combinatorics of polytopes.
 In this article, we put these equivariant manifolds into the framework of  orbit configuration spaces, especially for the cases of small covers and quasi-toric manifolds. In other words, we pay  much more attention to non-free orbit configuration space $F_{G^n_d}(M,k)$ for a  $dn$-dimensional $G^n_d$-manifold $\pi_d: M\longrightarrow P$ over a simple convex $n$-polytope $P$,  $d=1, 2$, where $M$ is a small cover and $G^n_1={\Bbb Z}_2^n$  when $d=1$, and a quasi-toric manifold and $G^n_2=T^n$ when $d=2$. We still expect that there is an essential connection between topology and geometry of $F_{G^n_d}(M,k)$ and combinatorics of $P$.

\subsection{Statement of main results} Our first result is an explicit formula for the Euler
characteristic of $F_{G^n_d}(M,k)$ in terms of the $h$-vector $(h_0,
h_1, ..., h_n)$ of $P$, and in particular,
$\chi(F_{T^n}(M,k))=\chi(F(M,k))$ if $d=2$. Let ${\bf
h}_P(t)=h_0+h_1t+\cdots+h_nt^n$ be a polynomial in ${\Bbb Z}[t]$.
Then our result is stated as follows.

\begin{thm} \label{main thm}
Let $\pi_d: M\longrightarrow P$ be a $dn$-dimensional $G^n_d$-manifold over a simple convex $n$-polytope $P$ where $d=1,2$.   Then  the Euler characteristic of $F_{G_d^n}(M, k)$ is
$$\chi(F_{G_d^n}(M, k))=
\begin{cases} (-1)^{kn}\sum\limits_{I=(k_1,..., k_s)}\mathcal{C}_I\prod\limits_{i=1}^s{\bf h}_P(1-2^{k_i}) & \text{ if } d=1\\
\chi(F(M,k))=\sum\limits_{I=(k_1,..., k_s)}\mathcal{C}_I\big({\bf h}_P(1)\big)^s & \text{ if } d=2
\end{cases} $$
where $I=(k_1,..., k_s)$ runs over all partitions of $k$, $\mathcal{C}_I=\frac{k!(-1)^{k-s}}{r_1!r_2!\cdots
r_s!k_1k_2\cdots k_s}$ and $r_i$ is the  number of times that $k_i$ appears in $I$.
\end{thm}

Our method for proving this theorem is to investigate the combinatorial structure
on $F_{G_d^n}(M, k)$. As a consequence of this method, we can also give a formula for the
 Euler characteristic of a non-equivariant configuration space
$F(M, k)$ in terms of a polynomial of $\chi(M)$.

\begin{thm}\label{non-eq}
Let $M$ be a compact triangulated homology $n$-manifold. Then
$$\chi(F(M, k))=(-1)^{kn}\prod_{i=0}^{k-1}(\chi(M)-i)=(-1)^{kn}k!{{\chi(M)}\choose k}.$$
\end{thm}

\begin{rem} The above formula can be rewritten as
$$
1+\sum\limits_{k=1}^\infty \frac{(-1)^{kn}\chi(F(M, k))}{k!}t^k=1+\sum_{k=1}^\infty\binom{\chi(M)}{k}t^k=(1+t)^{\chi(M)}.
$$
 Let $n$ be even. Then we obtain the F\'elix-Thomas formula~\cite[Theorem B]{FT}. Hence Theorem~\ref{non-eq} generalizes the F\'elix-Thomas formula.
 \end{rem}

Next we shall be concerned with the homotopy type of $F_{G_d^n}(M,k)$ for   $n\geq 1$.  We first consider the case  $n=1$. In this case, we obtain

\begin{thm}\label{theorem1.3}
Let $\pi_d: M\longrightarrow P$ be a $d$-dimensional $G_d^1$-manifold over $P$. Then, when $d=1$,  $F_{\mathbb Z_2}(M,k)$
has the same  homotopy type as  $k!2^{k-2}$ points, and when $d=2$, $F_{S^1}(M, k)$ has the same homotopy type as a disjoint union of $k!$ copies of $T^{k-2}$.
\end{thm}

\begin{rem}
The classical configuration spaces on the circle is related to hyperbolic Dehn fillings~\cite{YNK}. The orbit version might give some additional information.
\end{rem}

For general $M$ and $k$, the spaces $F_{G_d^n}(M,k)$  can be expressed as an intersection of the subspaces of $M^{\times k}$ which are homeomorphic to $M^{\times (k-2)}\times F_{G_d^n}(M,2)$ under coordinate permutations (see Proposition~\ref{constr}). Thus the study on $F_{G_d^n}(M,2)$ is the first step for the general cases. In this article we  focus on this case  to give an experimental investigation of the homotopy type of $F_{G_d^n}(M,2)$, where the
combinatorial methods successfully overcome the technical difficulties in this case. The spaces
$F_{G_d^n}(M,k)$  for general $k$ will be explored in our subsequent work.
By the reconstruction of small covers and quasi-toric manifolds,
 we are able to determine an equivariant  strong deformation retract of $F_{G_d^n}(M,2)$ in terms of the combinatorial data from $P$ via $\pi_d$.
 The result is stated as follows.

 \begin{thm} \label{theorem1.4}
 Let $\pi_d: M\longrightarrow P$ be a $dn$-dimensional $G_d^n$-manifold over a simple convex polytope $P$. Then
 there is an equivariant strong deformation retraction of $F_{G_d^n}(M,2)$ onto
$$X_d(M)=\bigcup_{F_1, F_2\in \mathcal{F}(P)\
\atop F_1\cap F_2=\emptyset} (\pi_d^{-1})^{\times 2}(F_1\times F_2)$$
  where  $\mathcal{F}(P)$ is the set of all faces of $P$.
  \end{thm}
The equivariant strong deformation retract $X_d(M)$ in Theorem~\ref{theorem1.4} plays an important role on further
studying  the algebraic topology of $F_{G_d^n}(M,2)$. 
The dual cell decomposition of the union $\bigcup_{F_1, F_2\in
\mathcal{F}(P)\ \atop F_1\cap F_2=\emptyset} (F_1\times F_2)$ as a
polyhedron is a simplicial complex $K_P$, which indicates  the
intersection property of all submanifolds $(\pi_d^{-1})^{\times
2}(F_1\times F_2), F_1, F_2\in \mathcal{F}(P)$ with $F_1\cap
F_2=\emptyset$.
     As shown in section~\ref{MV},  $X_d(M)$ with $K_P$ together  determines  a Mayer-Vietoris spectral sequence
$E^1_{p,q}(K_P, d), ..., E^\infty_{p,q}(K_P, d)$ with ${\Bbb Z}$ coefficients, which converges to $H_*(X_d(M))\cong H_*(F_{G_d^n}(M,2))$. Namely
$$H_i(F_{G_d^n}(M,2))\cong \sum_{p+q=i}E_{p,q}^\infty(K_P,d)$$
 (see Theorem~\ref{mv-c}). We shall prove that
\begin{thm}\label{main-betti} Assume that $\pi_d:M\longrightarrow P$
is a $dn$-dimensional $G_d^n$-manifold over a simple convex polytope $P$. Then  there is the following isomorphism
$$E_{p,q}^2(K_P, 1)\otimes{\Bbb Z}_2\cong E_{p,2q}^2(K_P, 2)\otimes{\Bbb Z}_2.$$
\end{thm}

For the case $d=2$, we also have

 \begin{thm}\label{eq-coh} Let
$\pi_2:M\longrightarrow P$ be a $2n$-dimensional quasi-toric
manifold over a simple convex polytope $P$. Then the associated
Mayer-Vietoris spectral sequence of the space $X_2(M)$ collapses at
the $E^2$ term, that is,
$$H_i(F_{T^n}(M,2))\cong \sum_{p+q=i}E_{p,q}^2(K_P,2).$$
\end{thm}




Furthermore,  as a consequence of Theorems~\ref{theorem1.4}--\ref{eq-coh},
we can also  determine:
 \begin{enumerate}
\item  the integral homology  of $F_{G_d^2}(M,2)$ and
\item the (mod~$2$) homology  of $F_{G_d^n}(M,2)$ for  $M$ to be a $G_d^n$-manifold over an $n$-simplex  $\Delta^n$.
In this case, $M$ is one of ${\Bbb R}P^n, {\Bbb C}P^n$ or
$\overline{{\Bbb C}P}^n$ (see \cite[Page 426]{bp}).
\end{enumerate}

The article is organized as follows. In section~\ref{defn}, we give a brief review on the notions of small covers,
quasi-toric manifolds and (real) moment-angle manifolds and investigate basic constructions and properties of their
orbit configuration spaces. Then we calculate the Euler characteristic of the orbit configuration spaces for small covers and quasi-toric manifolds in section~\ref{euler}, where Theorem~\ref{main thm} is Theorem~\ref{main-1 thm} and the proof of Theorem~\ref{non-eq} is given in subsection~\ref{subsection3.4}. In section~\ref{homotopy type}, we study the homotopy type of $F_{G_d^1}(M,k)$ and $F_{G_d^n}(M,2)$ with giving the proofs of Theorems~\ref{theorem1.3} and~\ref{theorem1.4}. As an application of Theorem~\ref{theorem1.4}, we 
prove Theorems~\ref{main-betti} and~\ref{eq-coh} in section~\ref{b-eq}.  In section~\ref{ca}, we compute the integral homology of $F_{G_d^2}(M,2)$ and the (mod~$2$) homology of $F_{G_d^n}(M,2)$ for the $G_d^n$-manifold $M$ over an $n$-simplex. The Mayer--Vietoris spectral sequence will be one of major tools for our computations and so we give a review on the Mayer-Vietoris spectral sequence in section~\ref{MV} as an appendix.

\section{$G_d^n$-manifolds and (real) moment-angle manifolds over simple convex polytopes and their orbit configuration spaces}\label{defn}

\subsection{$G_d^n$-manifolds and (real) moment-angle manifolds over simple convex polytopes}
Following \cite{dj}, let $P$ be a simple convex $n$-polytope, and let $G_d^n$ be the 2-torus ${\Bbb Z}_2^n$ of rank $n$ if $d=1$, and the torus $T^n$ of rank $n$ if $d=2$.
A $dn$-dimensional {\em $G_d^n$-manifold over $P$},  $\pi_d: M\longrightarrow P$,  is a smooth
closed $dn$-dimensional manifold $M$ with a locally standard $G_d^n$-action such that the orbit space is $P$.
 A $G_d^n$-manifold $\pi_d: M\longrightarrow P$ is called a {\em small cover} if $d=1$ and a {\em quasi-toric manifold} if $d=2$. We know from \cite{dj} that each $G_d^n$-manifold $\pi: M \longrightarrow P$ determines a characteristic
function $\lambda_d$  on $P$, defined by mapping all facets (i.e.,
$(n-1)$-dimensional faces) of $P$ to nonzero elements of $R_d^n$ such that $n$ facets meeting
at each vertex are mapped to  a basis of $R_d^n$ where $R_d=\begin{cases}
{\Bbb Z}_2 & \text{ if } d=1\\
{\Bbb Z} & \text{ if } d=2.
\end{cases}$ Conversely, the pair $(P, \lambda_d)$ can be reconstructed to the $M$ as follows:
first $\lambda_d$
gives the following  equivalence relation $\sim_{\lambda_d}$ on $P\times G_d^n$
\begin{equation}\label{equiv}
(x, g)\sim_{\lambda_d} (y, h)\Longleftrightarrow \begin{cases} x=y, g=h & \text{ if } x\in \text{\rm int}(P)\\
x=y, g^{-1}h\in G_F &\text{ if } x\in \text{\rm int}F\subset \partial P
\end{cases}\end{equation}
then the quotient space $P\times G_d^n/\sim_{\lambda_d}$ is equivariantly homeomorphic to the $M$, where $G_F$ is explained as follows:  for each point $x\in \partial P$, there exists a unique face $F$ of $P$ such that $x$ is in its relative interior. If $\dim F=l$, then there are $n-l$ facets, say
$F_{i_1}, ..., F_{i_{n-l}}$, such that $F=F_{i_1}\cap\cdots \cap
F_{i_{n-l}}$, and furthermore, $\lambda_d(F_{i_1}), ..., \lambda_d(F_{i_{n-l}})$
determine a subgroup of rank $n-l$ in $G_d^n$, denoted by
$G_F$.
This reconstruction of $M$ tells us that any topological invariant of $\pi_d: M\longrightarrow P$ can be determined by $(P, \lambda_d)$.  Davis and Januszkiewicz showed that $\pi_d: M\longrightarrow P$ has
a very beautiful algebraic topology in terms of $(P, \lambda_d)$. For example, the equivariant cohomology with $R_d$ coefficients of
 $\pi_d: M\longrightarrow P$ is isomorphic to the Stanley--Reisner face
ring of $P$, and the mod 2 Betti numbers $(b_0^{{\Bbb Z}_2}, b_1^{{\Bbb Z}_2}, ..., b_n^{{\Bbb Z}_2})$ of $M$ for $d=1$ and the Betti numbers $(b_0, b_2, ..., b_{2n})$ of $M$ for $d=2$ agree with the
$h$-vector $(h_0, h_1, ..., h_n)$ of $P$.

In addition, associated with a simple convex $n$-polytope $P$ with $m$ facets $F_1, ..., F_m$, Davis and Januszkiewicz also introduced a  $G_d^m$-manifold $\mathcal{Z}_{P, d}$ of dimension $(d-1)m+n$ over $P$ as follows: first define a map $\theta_d: \{F_1, ..., F_m\}\longrightarrow R_d^m$ by mapping $F_i\longmapsto e_i$ where $\{e_1, ..., e_m\}$ is the standard basis of $R_d^m$, and then use $\theta_d$ to give an equivalence relation $\sim_{\theta_d}$ on $P\times G_d^m$ as in (\ref{equiv}), so that the required  $G_d^m$-manifold $\mathcal{Z}_{P, d}$ is just the quotient $P\times G_d^m/\sim_{\theta_d}$ with a natural $G_d^m$-action having orbit space as $P$. Later on, Buchstaber and Panov \cite{bp} further studied the topology of $\mathcal{Z}_{P, d}$ as a submanifold in the polydisk $(D^d)^{\times m}$, and named it a real moment-angle manifold for $d=1$ and a moment-angle manifold for $d=2$. Note that a real moment-angle manifold and a moment-angle manifold are often denoted by ${\Bbb R}\mathcal{Z}_P$ and $\mathcal{Z}_P$, respectively.

As pointed out in \cite[Nonexample 1.22]{dj}, given a simple convex $n$-polytope $P$ with $n>3$,  there may not exist any $G_d^n$-manifold over $P$. However, there always exists a (real) moment-angle manifold over $P$.

When $P$ admits a characteristic function $\lambda_d$ (so there is a $G_d^n$-manifold $M^{dn}$ over $P$ reconstructed by $(P, \lambda_d)$), regarding $R_d^m$ as a free module generated by $\{F_1, ..., F_m\}$, the map $\lambda_d$ may linearly  extend to
a surjection $\widetilde{\lambda_d}: R_d^m\longrightarrow R_d^n$. Then the kernel of $\widetilde{\lambda_d}$ determines a subgroup $H$ of rank $m-n$ of $G_d^m$, which can freely act on $\mathcal{Z}_{P, d}$ such that the quotient manifold $\mathcal{Z}_{P, d}/H$ is exactly equivariantly homeomorphic to the $G_d^n$-manifold $M^{dn}$. Thus, the natural projection $\rho_d: \mathcal{Z}_{P, d}\longrightarrow M^{dn}$ is a fibration with fiber $G_d^{m-n}$. Davis and Januszkiewicz
showed in \cite{dj} that the Borel constructions $EG_d^m\times_{G_d^m}\mathcal{Z}_{P, d}$ and $EG_d^n\times_{G_d^n} M^{dn}$ are homotopy-equivalent.

For more details of these equivariant manifolds above with many
interesting developments and applications, e.g., see \cite{dj, bp,
bbcg, cl, cms, cps, ifm, lt, ly, m, ms, u}.

\subsection{Basic constructions and properties of orbit configuration spaces of
$G_d^n$-manifolds and (real) moment-angle manifolds}
Let $\pi_d: M\longrightarrow P$ be a $dn$-dimensional $G_d^n$-manifold over a simple convex polytope $P$.
Then the product $\pi_d^{\times k}: M^{\times k}\longrightarrow P^{\times k}$ is also
a $dkn$-dimensional $(G_d^{n})^{\times k}$-manifold over a simple convex polytope $P^{\times k}$.

\begin{defn} Set
$$\widetilde{\Delta}(P^{\times k})=\bigcup_{1\leq i<j\leq
k}\Delta_{i,j}(P^{\times k})$$ where $\Delta_{i, j}(P^{\times
k})=\{(p_1, p_2, ..., p_k)\in P^{\times k}| p_i=p_j\}$.
 We call $\widetilde{\Delta}(P^{\times
k})$  the {\em weak diagonal} of $P^{\times k}$. Set
$$\Delta(P^{\times k})=\{(p,..., p)\in P^{\times k}| p\in P\}, $$
which is called the {\em strong diagonal} of $P^{\times k}$.
\end{defn}

 By definition, we have that $F(P, k)= P^{\times k}-\widetilde{\Delta}(P^{\times k})$. By the constructions of $M$, we see that $F_{G_d^n}(M, k)$  is the pullback from $\pi_d^{\times k}: M^{\times k}\longrightarrow P^{\times k}$  via the inclusion
$F(P,k)\hookrightarrow P^{\times k}$. So there is the following commutative diagram:
$$\begin{CD}
    F_{G_d^n}(M, k) @>>> M^{\times k}\\
  @V{\overline{\pi}_d^{\times k}}VV @VV{\pi_d^{\times k}}V\\
  F(P,k) @>>> P^{\times k}.
\end{CD}$$
We see easily that $F_{G_d^n}(M, k)\subset M^{\times k}$ is a
non-free orbit configuration space, and admits an  action of
$(G_d^n)^{\times k}$ such that the orbit space is exactly $F(P, k)$.

\begin{prop}\label{constr}
Let $\pi_d: M\longrightarrow P$ be a $dn$-dimensional $G_d^n$-manifold over a simple convex polytope $P$. Then
$$F_{G_d^n}(M, k)=\bigcap_{1\leq i<j\leq k}(M^{\times k}-(\pi_d^{\times k})^{-1}(\Delta_{i, j}(P^{\times k}))).$$
\end{prop}
\begin{proof}
The required result follows  by using the De Morgan formula.
\end{proof}

\begin{rem}
In Proposition~\ref{constr}, each $M^{\times k}-(\pi_d^{\times k})^{-1}(\Delta_{i, j}(P^{\times k}))$ is homeomorphic to $M^{\times (k-2)}\times
F_{G_d^n}(M, 2)$.
\end{rem}

\begin{rem}
Let $M\longrightarrow P$ be a quasi-toric manifold over $P$. As shown in \cite[Corollary 1.9]{dj}, there is a conjugation involution $\tau$ on $M$ such that its fixed point set $M^\tau$ is exactly a small cover over $P$. This means that there is still an involution on $F_{T^n}(M,k)$ such that its fixed point set is
$F_{{\Bbb Z}_2^n}(M^\tau,k)$.
\end{rem}

Let $p_d: \mathcal{Z}_{P, d}\longrightarrow P$ be the (real) moment-angle manifold over $P$ with $m$ facets.  Similarly, we see from the constructions of  $\mathcal{Z}_{P, d}$ that  $F_{G_d^m}(\mathcal{Z}_{P, d}, k)$ admits an action of $(G_d^m)^{\times k}$ and is the pullback from  $p_d^{\times k}: \mathcal{Z}_{P, d}^{\times k}\longrightarrow P^{\times k}$ via the inclusion
$F(P,n)\hookrightarrow P^{\times n}$, so there is a commutative diagram
$$\begin{CD}
  F_{G_d^m}(\mathcal{Z}_{P, d}, k) @>>> \mathcal{Z}_{P, d}^{\times k}\\
  @V{\overline{p}_d^{\times k}}VV  @VV{p_d^{\times k}}V\\
   F(P,k) @>>> P^{\times k}.
\end{CD}$$
Thus we have that
$$F_{G_d^m}(\mathcal{Z}_{P, d}, k)=\bigcap_{1\leq i<j\leq k}(\mathcal{Z}_{P, d}^{\times k}-(p_d^{\times k})^{-1}(\Delta_{i, j}(P^{\times k}))).$$
If we assume that there exists a $G_d^n$-manifold $\pi_d: M\longrightarrow P$ over $P$, then we know that $\mathcal{Z}_{P, d}$ is a
  principal $G_d^{m-n}$-bundle over $M$, denoted by $\rho_d: \mathcal{Z}_{P, d}\longrightarrow M$. Then we have that $p_d=\pi_d\circ \rho_d$. Furthermore, we have  the following commutative diagram:
$$\begin{CD}
  F_{G_d^m}(\mathcal{Z}_{P, d}, k) @>>> \mathcal{Z}_{P, d}^{\times k}\\
  @V{\overline{\rho}_d^{\times k}}VV  @VV{\rho_d^{\times k}}V\\
  F_{G_d^n}(M, k) @>>> M^{\times k}\\
  @V{\overline{\pi}_d^{\times k}}VV @VV{\pi_d^{\times k}}V\\
  F(P,k) @>>> P^{\times k}.
\end{CD}$$
Since  $\overline{\rho}_d^{\times k}: F_{G_d^m}(\mathcal{Z}_{P, d},
k)\longrightarrow F_{G_d^n}(M, k)$ is  a pull-back via  the
inclusion $F_{G_d^n}(M, k) \hookrightarrow M^{\times k}$, it is a
fibration with fiber $(G_d^{m-n})^{\times k}$. In the same way as in
\cite[4.1]{dj} and \cite[Proposition 6.34]{bp}, we have the
following homotopy-equivalent Borel constructions
$$E(G_d^m)^{\times k}\times_{(G_d^m)^{\times k}}F_{G_d^m}(\mathcal{Z}_{P, d}, k)\simeq E(G_d^n)^{\times k}\times_{(G_d^n)^{\times k}} F_{G_d^n}(M, k).$$
Thus we conclude that

\begin{prop} Given a simple convex $n$-polytope $P$ with $m$ facets, assume that  $\pi_d: M\longrightarrow P$ is a  $G_d^n$-manifold over $P$. Let $p_d: \mathcal{Z}_{P, d}\longrightarrow P$ be the $($real$)$ moment-angle manifold over $P$. Then the equivariant cohomologies of $F_{G_d^n}(M, k)$ and $F_{G_d^m}(\mathcal{Z}_{P, d}, k)$ are isomorphic, i.e.,
$$H^*_{(G_d^n)^{\times k}}(F_{G_d^n}(M, k))\cong H^*_{(G_d^m)^{\times k}}(F_{G_d^m}(\mathcal{Z}_{P, d}, k)).$$
\end{prop}

\section{Euler characteristic of $F_{G_d^n}(M, k)$}\label{euler ch}\label{euler}

The objective of this section is to calculate the Euler characteristic $\chi(F_{G_d^n}(M, k))$ for a $dn$-dimensional  $G_d^n$-manifold $\pi_d: M\longrightarrow P$.

\subsection{Euler characteristic of union--the inclusion-exclusion principle}
Suppose that $X_1,...,X_N$ are CW-complexes such that all their
possible nonempty intersections are subcomplexes of
$X_1\cup\cdots\cup X_N$. Let $\Delta^{N-1}$ be the abstract
$(N-1)$-simplex on vertex set $[N]=\{1,...,N\}$, i.e.,
$\Delta^{N-1}=2^{[N]}$ (the power set of $[N]$).  For each $a\in
2^{[N]}$, set
$$X_a=\begin{cases}
\bigcap_{i\in a}X_i & \text{ if } a\not=\emptyset\\
\bigcup_{i=1}^N X_i & \text{ if } a=\emptyset.
\end{cases}$$
Since each pair $(X_i, X_j)$ is an excisive couple of $X_i\cup X_j$,  we have the following well-known formula for euler characteristics.
\begin{prop}[Inclusion-exclusion principle]\label{in-ex}
$$\chi(X_\emptyset)=\sum_{a\in 2^{[N]}\
\atop a\not=\emptyset}(-1)^{|a|-1}\chi(X_a).$$
\end{prop}

 \subsection{The $h$-polynomial and the cell-vector of $P$}
Let $P$ be a simple convex $n$-polytope. The {\em $f$-vector} of $P$ is an integer vector $(f_0, f_1, ..., f_{n-1})$, where $f_i$ is the number of faces of $P$ of codimension $i+1$ (i.e., of dimension $n-i-1$). Then the {\em $h$-vector} of $P$ is the integer vector $(h_0, h_1, ..., h_n)$ defined from the following equation
\begin{equation}\label{f-h}
h_0+h_1t+\cdots+h_nt^n=(t-1)^n+f_0(t-1)^{n-1}+\cdots+ f_{n-2}(t-1)+f_{n-1}.
\end{equation}
The $f$-vector and the $h$-vector determine each other by Equation~(\ref{f-h}).

Let ${\bf h}_P(t)=h_0+h_1t+\cdots+h_nt^n$ be a polynomial in ${\Bbb Z}[t]$. We call ${\bf h}_P(t)$ the {\em $h$-polynomial} of $P$.
Given a finite CW-complex $X$ of dimension $l$, the {\em cell--vector} $c(X)$ of $X$ is the integer vector $(c_0, c_1, ..., c_l)$ where $c_i$ denotes the number of all $i$-cells in $X$. Each simple convex $n$-polytope $P$ has a natural cell decomposition such that  the interior $\text{int}F$ of an $i$-face $F$ of $P$ is an $i$-cell. Thus,
$$c(P)=(f_{n-1}, f_{n-2}, ..., f_1, f_0, 1)$$
where $f(P)=(f_0, f_1, ..., f_{n-1})$ is the $f$-vector of $P$.

\begin{lem}\label{part} Let $\pi_d:M\longrightarrow P$ be a $dn$-dimensional $G_d^n$-manifold over a simple convex $n$-polytope $P$. Then for a positive integer $\ell$,  the Euler characteristic of $(\pi_d^{\times \ell})^{-1}(\Delta(P^{\times \ell}))$ is
$$\chi((\pi_d^{\times \ell})^{-1}(\Delta(P^{\times \ell})))=
\begin{cases} {\bf h}_P(1-2^\ell) & \text{ if } d=1\\
{\bf h}_P(1)=\chi(M) & \text{ if } d=2.
\end{cases}$$
\end{lem}

\begin{proof}
Fix the cell decomposition of $P$ as above such that its cell-vector $c(P)=(f_{n-1}, f_{n-2}, ..., f_1, f_0, 1)$.
Let $F$ be a face of dimension $i$.  By \cite[Lemma 1.3]{dj}, we know that $\pi_d^{-1}(F)$ is still a $di$-dimensional $G_d^i$-manifold  over $F$, and in particular,   $\pi_d^{-1}(\text{\rm int}F)=G_d^i\times \text{\rm int}F$. When $d=1$,  $\pi_1^{-1}(\text{\rm int}F)$ is the disjoint union of $2^i$ copies of $\text{\rm int}F$. Since the strong diagonal $\Delta(F^{\times \ell})$ is combinatorially equivalent to $F$, $(\pi_1^{\times \ell})^{-1}(\Delta((\text{\rm int}F)^{\times\ell}))$ is the disjoint union of $2^{i\ell}$ copies of $\Delta((\text{\rm int}F)^{\times\ell})$.
Thus, the cell--vector of $(\pi_1^{\times \ell})^{-1}(\Delta(P^{\times \ell}))$ is
$$(f_{n-1}, 2^\ell f_{n-2}, ..., 2^{i\ell}f_{n-i-1}, ..., 2^{(n-2)\ell}f_1, 2^{(n-1)\ell}f_0, 2^{n\ell}).$$
Furthermore,
\begin{align*}
&\chi((\pi_1^{\times \ell})^{-1}(\Delta(P^{\times \ell})))\\
=& f_{n-1}-2^\ell f_{n-2}+\cdots+(-1)^i 2^{i\ell}f_{n-i-1}+\cdots+(-1)^{n-1}2^{(n-1)\ell}f_0+(-1)^n 2^{n\ell}\\
=&{\bf h}_P(1-2^\ell)
\end{align*}
 by Equation~(\ref{f-h}).
When $d=2$, we see easily that $(\pi_2^{\times \ell})^{-1}(\Delta((\text{\rm int}F)^{\times\ell}))=T^{i\ell}\times \Delta((\text{\rm int}F)^{\times\ell})$.
Now, for $i>0$, we give a cell decomposition for each circle $S^1$ in $T^{i\ell}$, with one $0$-cell and one $1$-cell. Then  $(\pi_2^{\times \ell})^{-1}(\Delta((\text{\rm int}F)^{\times\ell}))$ contains ${{i\ell}\choose j}$ cells of dimension-$(i+j)$ for $0\leq j\leq i\ell$. Thus,
all $i$-cells of $P$ contribute ${{i\ell}\choose j}f_{n-i-1}$ cells of dimension-$(i+j)$ in $(\pi_2^{\times \ell})^{-1}(\Delta(P^{\times \ell}))$ where $0\leq j\leq i\ell$. Since $\sum_{j=0}^{i\ell}(-1)^j{{i\ell}\choose j}=0$ for every $i>0$, by a direct calculation we have
\begin{align*}
&\chi((\pi_2^{\times \ell})^{-1}(\Delta(P^{\times \ell})))= f_{n-1}={\bf h}_P(1)=\chi(M)
\end{align*}
as desired.
\end{proof}

\subsection{Subgraphs of $\mathcal{K}_k$ and partitions of $k$ and $[k]$ }

 Let $\mathcal{K}_k$ be  the complete graph of degree $k-1$, which contains $k$ vertices and ${k\choose 2}$ edges.
 We label $k$ vertices of $\mathcal{K}_k$ by $1, ..., k$ respectively, and ${k\choose 2}$ edges by pairs $(i, j), 1\leq i<j\leq k$, respectively.
 Thus we may identify  $\mathcal{K}_k$  with  the 1-skeleton of the abstract $(k-1)$-simplex $\Delta^{k-1}$ on vertex set $[k]=\{1,..., k\}$. Obviously, $\Delta^{k-1}=2^{[k]}$, the power set of $[k]$.

\begin{defn}
A subgraph $\Gamma$ of $\mathcal{K}_k$ is said to be {\em vertex-full} if the vertex set of $\Gamma$ is $[k]$.
\end{defn}
By ${\bf VF}(\mathcal{K}_k)$ we denote the set of all vertex-full subgraphs  of $\mathcal{K}_k$.

\begin{lem}\label{1-1}
There is a one-to-one correspondence between all subsets of the power set $2^{[[k]]}$ and all vertex-full subgraphs of ${\bf VF}(\mathcal{K}_k)$, where $[[k]]=\{(i,j)| 1\leq i<j\leq k\}$.
\end{lem}

\begin{proof}
 Each vertex-full subgraph $\Gamma$ of $\mathcal{K}_k$ uniquely determines a subset $E(\Gamma)$ of $2^{[[k]]}$, where $E(\Gamma)$ denotes the set of all edges of $\Gamma$. Note that the discrete  subgraph $[k]$ of $\mathcal{K}_k$ corresponds to the empty set $\emptyset$ of
$2^{[[k]]}$. Conversely, let the empty set $\emptyset$ of $2^{[[k]]}$ correspond to the discrete  subgraph $[k]$ of $\mathcal{K}_k$. Each nonempty subset  of $2^{[[k]]}$ determines a unique subgraph $\Gamma$ of $\mathcal{K}_k$. If the vertex set of $\Gamma$ does not cover $[k]$, then we can add those missing vertices as one-point subgraphs to $\Gamma$ to give the required vertex-full subgraph of $\mathcal{K}_k$.
\end{proof}

\begin{defn} Given a vertex-full subgraph $\Gamma$ in ${\bf VF}(\mathcal{K}_k)$, define
$$
\Delta_\Gamma(P^{\times k})=
\begin{cases}
\bigcap_{(i,j)\in E(\Gamma)}\Delta_{i,j}(P^{\times k}) & \text{ if } E(\Gamma)\neq\emptyset\\
P^{\times k} & \text{ if } E(\Gamma)=\emptyset
\end{cases}
$$
where $E(\Gamma)$ denotes the set of all edges of $\Gamma$. Generally,  $\Gamma$ may not be connected. By $C(\Gamma)$ we denote the set of all connected subgraphs of $\Gamma$. Write $C(\Gamma)=\{\Gamma_1, ..., \Gamma_s\}$. Then $\Gamma=\coprod_{k=1}^s\Gamma_k$ (a disjoint union of $\Gamma_1, ..., \Gamma_s$).
\end{defn}

\begin{lem}\label{product}
Let $\Gamma$ be a vertex-full subgraph in ${\bf VF}(\mathcal{K}_k)$ with $C(\Gamma)=\{\Gamma_1, ..., \Gamma_s\}$.
Then $$\Delta_\Gamma(P^{\times k})=\prod_{l=1}^s\Delta(P^{\times|V(\Gamma_l)|})$$
where $V(\Gamma_l)$ denotes the vertex set of $\Gamma_l$.
\end{lem}

\begin{proof}
Obviously, if $\Gamma=[k]$, then the required equality holds. Suppose that $\Gamma\not=[k]$. For each component $\Gamma_l$ of $\Gamma$,
$\bigcap_{(i,j)\in E(\Gamma_l)}\Delta_{i,j}(P^{\times k})$ is combinatorially equivalent to $P^{\times (k-|V(\Gamma_l)|)}\times \Delta(P^{\times |V(\Gamma_l)|})$.  Thus
\begin{align*}
\Delta_\Gamma(P^{\times k})=\bigcap_{l=1}^s\bigcap_{(i,j)\in E(\Gamma_l)}\Delta_{i,j}(P^{\times k})=\prod_{l=1}^s\Delta(P^{\times|V(\Gamma_l)|})
\end{align*}
as desired.
\end{proof}

Recall that a {\em partition} of $k$ is an unordered sequence $(k_1, ..., k_s)$ of positive integers with sum $k$, and a {\em partition} of $[k]$ is an unordered sequence of nonempty subsets of $[k]$ which are pairwise disjoint and whose union is $[k]$.
Clearly, every  vertex-full subgraph $\Gamma=\coprod_{l=1}^s\Gamma_l$ of ${\bf VF}(\mathcal{K}_k)$ gives a partition $(|V(\Gamma_1)|, ..., |V(\Gamma_s)|)$  of $k$, denoted by $k(\Gamma)$. In addition, each vertex-full subgraph  of ${\bf VF}(\mathcal{K}_k)$ also determines a  partition $(V(\Gamma_1), ..., V(\Gamma_s))$ of $[k]$.

\begin{lem}\label{coe}
Let  $I=(k_1, ..., k_s)$ be a partition of $k$. Then the number of those combinatorially equivalent vertex-full subgraphs $\Gamma$ with $k(\Gamma)=I$ of ${\bf VF}(\mathcal{K}_k)$ is
$${{k!}\over {k_1!\cdots k_s!r_1!\cdots r_s!}}$$
where $r_i$ denotes the  number of times that $k_i$ appears in $I$.
\end{lem}

\begin{proof}
Obviously, those combinatorially equivalent vertex-full subgraphs $\Gamma$ with $k(\Gamma)=I$ of ${\bf VF}(\mathcal{K}_k)$ bijectively correspond to those
partitions $(a_1, ..., a_s)$ with $|a_l|=k_l$ of $[k]$. The desired number then follows from an easy argument.
\end{proof}

\subsection{Calculation of Euler characteristic} \label{subsection3.4}
Now let us calculate  $\chi(F_{G_d^n}(M, k))$ for a  $G_d^n$-manifold
 $\pi_d: M\longrightarrow P$ over a simple convex polytope $P$.

 \begin{defn}
 Let  $I=(k_1, ..., k_s)$ be a partition of $k$. Define
 $$\mathcal{C}_I=\sum_{\Gamma\in {\bf VF}(\mathcal{K}_k)\
 \atop k(\Gamma)=I} (-1)^{|E(\Gamma)|}.$$
 \end{defn}

 \begin{thm} \label{main-1 thm}
Let $\pi_d: M\longrightarrow P$ be a $dn$-dimensional $G_d^n$-manifold over a simple convex polytope.   Then
$$\chi(F_{G_d^n}(M, k))=
\begin{cases}
(-1)^{kn}\sum_{I=(k_1,..., k_s)}\mathcal{C}_I\prod_{l=1}^s{\bf h}_P(1-2^{k_l}) &\text{ if } d=1\\
\chi(F(M, k)) & \text{ if } d=2
\end{cases} $$
where $I=(k_1,..., k_s)$ runs over all partitions of $k$.
\end{thm}

 \begin{proof}
 First, we calculate $\chi((\pi_d^{\times k})^{-1}(\widetilde{\Delta}(P^{\times k})))$ by Proposition~\ref{constr}.
Since $2^{[[k]]}$ is combinatorially equivalent to $2^{[N]}$ where $N={k\choose 2}$,  we have that
\begin{align*}
&\chi((\pi_d^{\times k})^{-1}(\widetilde{\Delta}(P^{\times k})))\\
=&\sum_{\Gamma=\coprod_{l=1}^s\Gamma_l\in {\bf VF}(\mathcal{K}_k)\
\atop E(\Gamma)\neq\emptyset}(-1)^{|E(\Gamma)|-1}\chi((\pi_d^{\times k})^{-1}(\Delta_\Gamma(P^{\times k}))) \text{
(by Lemma~\ref{1-1} and Prop.~\ref{in-ex}})\\
=& \sum_{\Gamma=\coprod_{l=1}^s\Gamma_l\in {\bf VF}(\mathcal{K}_k)\
\atop E(\Gamma)\neq\emptyset}(-1)^{|E(\Gamma)|-1}\chi((\pi_d^{\times k})^{-1}(\prod_{l=1}^s\Delta(P^{\times|V(\Gamma_l)|})))
\ (\text{by Lemma~\ref{product}})\\
=& \sum_{\Gamma=\coprod_{l=1}^s\Gamma_l\in {\bf VF}(\mathcal{K}_k)\
\atop E(\Gamma)\neq\emptyset}(-1)^{|E(\Gamma)|-1}\prod_{l=1}^s\chi((\pi_d^{\times |V(\Gamma_l)|})^{-1}(\Delta(P^{\times|V(\Gamma_l)|})))\\
=&\begin{cases}
 \sum_{\Gamma=\coprod_{l=1}^s\Gamma_l\in {\bf VF}(\mathcal{K}_k)\
\atop E(\Gamma)\neq\emptyset}(-1)^{|E(\Gamma)|-1}\prod_{l=1}^s{\bf h}_P(1-2^{|V(\Gamma_l)|}) & \text{ if } d=1\\
\sum_{\Gamma=\coprod_{l=1}^s\Gamma_l\in {\bf VF}(\mathcal{K}_k)\
\atop E(\Gamma)\neq\emptyset}(-1)^{|E(\Gamma)|-1}\chi(M)^s & \text{ if } d=2
\end{cases} (\text{by Lemma~\ref{part}.})
\end{align*}
Since each vertex-full subgraph $\Gamma=\coprod_{l=1}^s\Gamma_l$ of $\mathcal{K}_k$ corresponds to a unique partition $k(\Gamma)=(|V(\Gamma_1)|, ...,
|V(\Gamma_s)|)$ of $k$, we further have that
\begin{align*}\label{euler}
\chi((\pi_d^{\times k})^{-1}(\widetilde{\Delta}(P^{\times k})))
=& \begin{cases} -\sum_{I=(k_1,..., k_s)\
\atop I\not=(1, ..., 1)}\mathcal{C}_I\prod_{l=1}^s{\bf h}_P(1-2^{k_l}) & \text{ if } d=1\\
 -\sum_{I=(k_1,..., k_s)\
\atop I\not=(1, ..., 1)}\mathcal{C}_I\chi(M)^s  & \text{ if } d=2
\end{cases}
\end{align*}
where $I$ runs over those partitions except for $(1,..., 1)$ of $k$.  A direct calculation gives that $\mathcal{C}_{(1,...,1)}=1$,  so $$\chi(M^{\times k})=\begin{cases}
\mathcal{C}_{(1,...,1)}\big({\bf h}_P(-1)\big)^k & \text{ if } d=1\\
\mathcal{C}_{(1,...,1)}\big({\bf h}_P(1)\big)^k=\mathcal{C}_{(1, ..., 1)}\chi(M)^k & \text{ if } d=2
\end{cases}$$ by Lemma~\ref{part}.

Next,   we conclude that

\begin{align*}
&\chi(F_{G_d^n}(M,k))\\
= & (-1)^{dkn}\big(\chi(M^{\times k})-\chi((\pi_d^{\times k})^{-1}(\widetilde{\Delta}(P^{\times k})))\big)
\text{ (by Lefschetz duality theorem)}\\
=&
\begin{cases}
 (-1)^{kn}\big(\chi(M^{\times k})+\sum_{I=(k_1,..., k_s)\
\atop I\not=(1, ..., 1)}\mathcal{C}_I\prod_{l=1}^s{\bf h}_P(1-2^{k_l})\big)& \text{ if } d=1 \\
\chi(M^{\times k})+\sum_{I=(k_1,..., k_s)\
\atop I\not=(1, ..., 1)}\mathcal{C}_I\chi(M)^s  & \text{ if } d=2
\end{cases}\\
=&
\begin{cases}
 (-1)^{kn}\sum_{I=(k_1,..., k_s)}\mathcal{C}_I\prod_{l=1}^s{\bf h}_P(1-2^{k_l}) & \text{ if } d=1 \\
\sum_{I=(k_1,..., k_s)}\mathcal{C}_I\chi(M)^s  & \text{ if } d=2.
\end{cases}\\
\end{align*}
We can employs the above way  to the non-equivariant case for $d=2$, so that we may obtain that
\begin{align*}
\chi(F(M, k))=& (-1)^{2kn}\big(\chi(M^{\times k})-\chi(\widetilde{\Delta}(M^{\times k}))\big)\\
=& \chi(M^{\times k})-\chi(\widetilde{\Delta}(M^{\times k})\\
=& \sum_{I=(k_1,..., k_s)}\mathcal{C}_I\chi(M)^s\\
=& \chi(F_{T^n}(M,k))
\end{align*}
as desired.
\end{proof}

We know from~\cite{mu}  that the Lefschetz duality theorem holds for compact triangulated homology manifolds. Thus, using the proof method of Theorem~\ref{main-1 thm}, we can obtain the following formula for more general non-equivariant configuration spaces.

\begin{thm}\label{non-eq1}
Let $M$ be a compact triangulated homology $n$-manifold. Then
$$\chi(F(M,k))=(-1)^{kn}\sum_IC_I\chi(M)^s$$
where $I=(k_1, ..., k_s)$ runs over all partitions of $k$.
\end{thm}

In Theorem~\ref{non-eq1}, if we further write $$\chi(F(M,k))=(-1)^{kn}\sum_IC_I\chi(M)^s=(-1)^{kn}\sum_{s=1}^k\big(\sum_{I=(k_1, ..., k_s)}\mathcal{C}_I\big)\chi(M)^s$$
then we see that $\chi(F(M,k))$ is actually a polynomial (with ${\Bbb Z}$ coefficients) of $\chi(M)$ of degree $k$. By $g(t)$ we denote this polynomial in ${\Bbb Z}[t]$. We first complete the proof of Theorem~\ref{non-eq}.

\begin{proof}[Proof of Theorem~\ref{non-eq}]
It suffices to show that $g(t)=(-1)^{kn}\prod_{l=0}^{k-1}(t-l)$.
For $0\leq l<n$, choose $M$ as a set consisting of $l$ points. Then $M$ is a 0-dimensional manifold  if $0< l<k$, and  a empty set (or $-1$-dimensional manifold) if $l=0$. Thus  $$\chi(M)=\begin{cases}
l & \text{ if } 0<l<k\\
0 & \text{ if } l=0.
\end{cases}$$
By using the pigeonhole principle, since $l<k$, we see that $$F(M,k)=\{(x_1,\cdots,x_k)\in M^{\times k}|x_i\neq
x_j \text{ for } i\neq j\}$$ must be empty, so $\chi(F(M,k))=0$. This implies that $g(l)=0$ for $0\leq l<k$, and thus each $l$ is a root of $g(t)$. Furthermore, we can write $g(t)=(-1)^{kn}c\prod_{l=0}^{k-1}(t-l)$ where $c$ is a constant number. Since $\mathcal{C}_{(1,..., 1)}=1$, we conclude that $c$ must be 1. This completes the proof.
\end{proof}

\begin{cor}\label{cor}
$$\mathcal{C}_{(k)}=(-1)^{k-1}(k-1)!.$$
\end{cor}
\begin{proof}
This can be obtained by comparing the coefficients of $\chi(M)$ on both sides of the following equality
$$\sum_{s=1}^k(\sum_{I=(k_1, ..., k_s)}C_I)\chi(M)^s=\prod_{l=0}^{k-1}(\chi(M)-l).$$
\end{proof}

Finally, to complete the proof of Theorem~\ref{main thm}, it remains to determine the number $\mathcal{C}_I$ for every partition $I$ of $k$.

\begin{prop}
Let $I=(k_1,..., k_s)$ be a partition of $k$. Then
$$\mathcal{C}_I=\frac{k!(-1)^{k-s}}{r_1!r_2!\cdots r_s!k_1k_2\cdots k_s}$$
where $r_l$ denotes the number of times  that $k_l$ appears in $I$.
\end{prop}
\begin{proof}
Let $\mathcal{A}_I$ denote the set of those vertex-full subgraphs $\Gamma$ with $k(\Gamma)=I$ of ${\bf VF}(\mathcal{K}_k)$, all of which are not combinatorially equivalent to each other.
 Then we have that
\begin{align*}
\mathcal{C}_I=&\sum_{\Gamma\in {\bf VF}(\mathcal{K}_k)\
\atop k(\Gamma)=I}(-1)^{|E(\Gamma)|}\\
=& {{k!}\over {k_1!\cdots k_s!r_1!\cdots r_s!}}\sum_{\Gamma=\coprod_{l=1}^s\Gamma_l\in \mathcal{A}_I\
\atop |V(\Gamma_l)|=k_l}\prod_{l=1}^s(-1)^{|E(\Gamma_l)|} \text{ (by Lemma~\ref{coe})}\\
=&  {{k!}\over {k_1!\cdots k_s!r_1!\cdots r_s!}}\prod_{l=1}^s\sum_{\Gamma_l}(-1)^{|E(\Gamma_l)|}\\
=&  {{k!}\over {k_1!\cdots k_s!r_1!\cdots r_s!}}\prod_{l=1}^s \mathcal{C}_{(k_l)}\\
=&  {{k!}\over {k_1!\cdots k_s!r_1!\cdots r_s!}}\prod_{l=1}^s (-1)^{k_l-1}(k_l-1)! \text{ (by Corollary~\ref{cor})}\\
=& \frac{k!(-1)^{k-s}}{r_1!r_2!\cdots r_s!k_1k_2\cdots k_s}
\end{align*}
as desired.
\end{proof}

\begin{example}
By the formula of Theorem~\ref{main thm},  we have that
$$\chi(F_{{\Bbb Z}_2^n}(M, 2))=\big({\bf h}_P(-1)\big)^2-{\bf h}_P(-3)$$
and
$$\chi(F_{{\Bbb Z}_2^n}(M, 3))=(-1)^{3n}\Big(\big({\bf h}_P(-1)\big)^3-3{\bf h}_P(-1){\bf h}_P(-3)+2{\bf h}_P(-7)\Big)$$
where $M$ is a small cover over $P$.
\end{example}

\begin{cor}\label{real ma}
Given a simple convex $n$-polytope $P$ with $m$ facets, assume that there exists a small cover over $P$. Then
$$\chi(F_{{\Bbb Z}_2^m}(\mathcal{Z}_{P,1}, k))=(-1)^{kn}2^{(m-n)k}\sum\limits_{I=(k_1,..., k_s)}\frac{k!(-1)^{k-s}}{r_1!r_2!\cdots
r_s!k_1k_2\cdots k_s}\prod\limits_{l=1}^s{\bf h}_P(1-2^{k_l})$$
where $I=(k_1,..., k_s)$ runs over all partitions of $k$, and $r_l$ is the number of times that $k_l$ appears in $I$.
\end{cor}

\begin{proof}
Let $M$ be a small cover over $P$. Then $F_{{\Bbb Z}_2^m}(\mathcal{Z}_{P,1}, k)$ is a principal $({\Bbb Z}_2^{m-n})^{\times k}$-bundle over
$F_{{\Bbb Z}_2^n}(M,k)$. By \cite[p. 86, (1.5)(d)]{ap} we have that
$$\chi(F_{{\Bbb Z}_2^m}(\mathcal{Z}_{P,1}, k))=|({\Bbb Z}_2^{m-n})^{\times k}|\chi(F_{{\Bbb Z}_2^n}(M,k)).$$
Moreover, the required result follows from Theorem~\ref{main thm}.
\end{proof}

\begin{rem}
 It should be interesting to give an explicit formula of  $\chi(F_{{\Bbb Z}_2^m}(\mathcal{Z}_{P,1}, k))$ without the existence assumption of a small cover over $P$ in Corollary~\ref{real ma}.
\end{rem}

\begin{prop}
Let $P$ be a simple convex polytope with $m$ facets. Then $$\chi(F_{T^m}(\mathcal{Z}_{P,2}, k))=0.$$
\end{prop}

\begin{proof}
We know from \cite[Proposition 7.29]{bp} that the diagonal circle subgroup of $T^m$ acts freely on $\mathcal{Z}_{P,2}$, so the diagonal circle subgroup of $(T^m)^{\times k}$ also acts freely on $F_{T^m}(\mathcal{Z}_{P,2}, k)$. Therefore, $F_{T^m}(\mathcal{Z}_{P,2}, k)$ admits a principal $S^1$-bundle structure, which induces that
$\chi(F_{T^m}(\mathcal{Z}_{P,2}, k))=0$ by \cite[p. 86, (1.5)(c)]{ap}.
\end{proof}

\begin{rem}
Buchstaber and Panov \cite{bp} expanded the construction of $\mathcal{Z}_{P,d}$ over a simple convex polytope $P$ to the case of general simplicial complex $K$. The resulting space denoted by $\mathcal{Z}_{K,d}$ is not a manifold in general, called the (real) moment-angle complex. When $d=2$,  $\mathcal{Z}_{K, 2}$ still admits a principal $S^1$-bundle structure, so the Euler characteristic of its orbit configuration space is zero. It should be also interesting  to give an explicit formula for the Euler characteristic of the orbit configuration space of $\mathcal{Z}_{K, 1}$ in terms of the combinatorial data of $K$. In addition,
it was showed in \cite{cl, u} that the Halperin-Carlsson conjecture holds for $\mathcal{Z}_{K,d}$ with the restriction free action. Naturally, we wish to know whether this is also true for the orbit configuration space of $\mathcal{Z}_{K, d}$.
\end{rem}

\section{Homotopy type of $F_{G_d^n}(M,k)$ for $n=1$ or $k=2$}\label{homotopy type}

Throughout the following, assume that $\pi_d: M\longrightarrow P$ is a $dn$-dimensional $G_d^n$-manifold over a simple convex $n$-polytope $P$.
    We see easily that $F(P, k)$ is disconnected if $n=1$ and path-connected if $n>1$ since $\Delta(P^{\times 2})$ is combinatorially equivalent to $P$.

\subsection{Homotopy type of $F_{G_d^1}(M, k)$}
\begin{thm}\label{theorem4.1}
Let $\pi_d: M\longrightarrow P$ be a $d$-dimensional $G_d^1$-manifold over $P$. Then, when $d=1$,  $F_{\mathbb Z_2}(M,k)$
has the same  homotopy type as  $k!2^{k-2}$ points, and when $d=2$, $F_{S^1}(M, k)$ has the same homotopy type as a disjoint union of $k!$ copies of $T^{k-2}$.
\end{thm}
\begin{proof}
   It is well-known that when $d=1$,  $M$ is  a circle $S^1$ with a reflection fixing two isolated points such that the orbit polytope $P$ is  a 1-dimensional simplex, and when $d=2$, $M$ is a 2-sphere $S^2$ with a rotation action of $S^1$, fixing two isolated points, such that the orbit polytope $P$ is also a 1-dimensional simplex. Since a 1-simplex is homeomorphic to the interval $[0,1]$,  we may identify $P^{\times k}$ with the cube $[0,1]^{\times k}$.
We then  see that each point $x=(x_1, ..., x_k)\in F(P,k)\subset [0,1]^{\times k}$ determines a unique permutation $(\sigma(1), ..., \sigma(k))$ of $[k]$ such that $x_i<x_j$ as long as $\sigma(i)<\sigma(j)$, where $\sigma\in {\bf S}_k$, and ${\bf S}_k$ is the symmetric group on $[k]$. In particular, all points on the path $\big((1-t)x_1+t\frac{\sigma(1)-1}{k-1}, ..., (1-t)x_k+t\frac{\sigma(k)-1}{k-1}\big), 0\leq t\leq 1$, are in $F(P,k)$,  and  they also determine the unique permutation $(\sigma(1), ..., \sigma(k))$ of $[k]$. If $x$ and $x'$ are two different points in $F(P,k)$ such that they determine two different permutations $\sigma, \sigma'$ in  ${\bf S}_k$, then there must be no path from $x$ to $x'$ in $F(P, k)$ since any path from $x$ to $x'$ in $P^{\times k}$ always passes the point of the form $(y_1, ..., y_k)$ with some $y_i=y_j$ for $i\not=j$.
 Define a homotopy $H:F(P,k)\times
[0,1]\rightarrow F(P,k)$ by
$$((x_1, .., x_k), t)\longmapsto \big((1-t)x_1+t\frac{\sigma(1)-1}{k-1}, ..., (1-t)x_k+t\frac{\sigma(k)-1}{k-1}\big).$$
Then we have that this homotopy $H$  is a deformation retraction of $F(P,k)$ onto
$k!$ points in $\mathcal{A}=\big\{(\frac{\sigma(1)-1}{k-1},...,\frac{\sigma(k)-1}{k-1})\big|\sigma\in {\bf S}_k\big\}\subset F(P,k)$. This means that $F(P,k)$ contains
$k!$ connected components $C_\sigma, \sigma\in {\bf S}_k$, each of which  may continuously collapse  to a point in $\mathcal{A}$. For each $\sigma$, since $0,1\in\{\frac{\sigma(1)-1}{k-1},...,\frac{\sigma(k)-1}{k-1}\}$ and $\{\frac{\sigma(1)-1}{k-1},...,\frac{\sigma(k)-1}{k-1}\}-\{0,1\}$ is in the open interval $(0,1)$,   $C_\sigma$ has also the deformation retract $R_\sigma$ that is homeomorphic to a $(k-2)$-dimensional open ball $B_\sigma$  in $F(P, k)$, and is contained in $\partial P^{\times k}$. Therefore, each $R_\sigma$ can be chosen in the interior of an $(k-2)$-face in
$P^{\times k}$, so by \cite[Lemma 4.1]{dj},
$$(\pi_d^{\times k})^{-1}(R_\sigma)=G_d^{k-2}\times R_\sigma=\begin{cases}
{\Bbb Z}_2^{k-2}\times R_\sigma & \text{ if } d=1\\
T^{k-2}\times R_\sigma & \text{ if } d=2
\end{cases}$$
which is homotopic to $2^{k-2}$ points if $d=1$ and $T^{k-2}$ if $d=2$. This completes the proof.
\end{proof}

\subsection{An equivariant strong deformation retract of $F_{G_d^n}(M,2)$ with $n>1$}\label{n=2}
Let $\mathcal{F}(P)$ denote the set of all faces of $P$.
\begin{lem}\label{defor re}
There is a strong deformation retraction $\Omega:F(P,2)\times
[0,1]\rightarrow F(P,2)$ of $F(P,2)$ onto
$$\mathcal{A}(P,2)=\bigcup_{F_1, F_2\in \mathcal{F}(P)\
\atop F_1\cap F_2=\emptyset}F_1\times F_2.$$
\end{lem}

\begin{proof}
First, we note that $F(P,2)=P\times P-\Delta(P\times P)$ is
path-connected since $n>1$ and $\dim \Delta(P\times P)=\dim P$.
Actually, $F(P,2)$ is homotopic to $S^{n-1}$ (see \cite{co1}). Since
$\Delta(P\times P)$ always contains interior points of $P\times P$,
we have that $F(P, 2)$ can continuously collapse onto
$\partial(P\times P)-\Delta(P\times P)$. Since $\partial(P\times
P)=\partial P\times P\cup P\times\partial P$, in a similar way as
above, we can obtain that both $\partial P\times P$ and
$P\times\partial P$ can further continuously collapse onto $\partial
P\times
\partial P-\Delta(\partial P\times \partial P)$. Now we see that $\partial P\times
\partial P-\Delta(\partial P\times \partial P)$ is the union of subsets of the following forms
$$F\times F-\Delta(F\times F), F\times F'-\Delta(\partial P\times \partial P), F\times F''$$
where $F, F', F''$ are facets of $P$ with  $F\cap F'\not=\emptyset$ and $F\cap F''=\emptyset$,
  and further we also see that $F\times F'-\Delta(\partial P\times\partial P)$ can be continuously shrunk to the union of subsets of the following forms
 $$(F\cap F')\times (F\cap F')-\Delta((F\cap F')^{\times 2}),  Q\times (F\cap F')-\Delta(\partial P\times\partial P),$$ $$ (F\cap F')\times Q'-\Delta(\partial P\times\partial P), Q_1\times Q'_1$$
 where $Q$ and $Q_1$ are facets of $F$, $Q'$ and $Q'_1$ are facets of  $F'$,  such that $Q\times (F\cap F')\not=\emptyset$, $(F\cap F')\times Q'\not=\emptyset$,  and $Q_1\cap Q'_1=\emptyset$.
 We continuous the above
process to  $F\times F-\Delta(F\times F)$, $(F\cap F')\times (F\cap
F')-\Delta((F\cap F')^{\times 2})$, $Q\times (F\cap
F')-\Delta(\partial P\times\partial P)$, $(F\cap F')\times
Q'-\Delta(\partial P\times\partial P)$,  and further repeat it
whenever possible.  This process must end  after a finite number of
steps, giving finally that $F(P,2)$ is homotopic to the union
$\bigcup_{F_1\times F_2\in \mathcal{F}(P)\ \atop F_1\cap
F_2=\emptyset}F_1\times F_2$. In particular, we also
see that this shrinking of $F(P,2)$ leaves all points of
$\bigcup_{F_1\times F_2\in \mathcal{F}(P)\ \atop F_1\cap
F_2=\emptyset}F_1\times F_2$ fixed, and for each face $Q\times Q'\subset F(P,2)$ of $P^{\times 2}$ with $Q\cap Q'\not=\emptyset$, there exists  a sequence of faces in $P\times P$
$$Q\times Q'\supseteq Q_1\times Q'_1\supset \cdots\supset Q_{l-1}\times Q'_{l-1}\supset Q_l\times Q'_l$$
with $Q_i\cap Q'_i\not=\emptyset$ for $i=1, ..., l-1$ and $Q_l\cap
Q'_l=\emptyset$.   This means that  there is a
strong deformation retraction $\Omega:F(P,2)\times [0,1]\rightarrow
F(P,2)$ of $F(P,2)$ onto $\bigcup_{F_1\times F_2\in \mathcal{F}(P)\
\atop F_1\cap F_2=\emptyset}F_1\times F_2$.
\end{proof}

\begin{rem}\label{hom-type}
Because $F(P, 2)$ has the homotopy type of $S^{n-1}$,
$\mathcal{A}(P,2)$ has also the same homotopy type as $S^{n-1}$.
\end{rem}

\begin{thm}\label{main re}
There is an equivariant strong deformation retraction of $F_{G_d^n}(M,2)$ onto
$$X_d(M)=\bigcup_{F_1, F_2\in \mathcal{F}(P)\
\atop F_1\cap F_2=\emptyset} (\pi_d^{-1})^{\times 2}(F_1\times F_2).$$
\end{thm}

\begin{proof}
Since $M=P\times G_d^n/\sim_{\lambda_d}$, we have that
$F_{G_d^n}(M,2)=F(P, 2)\times
G_d^{2n}/\sim_{\lambda_d\times\lambda_d}$, and $X_d(M)=\mathcal{A}(P,2)\times
G_d^{2n}/\sim_{\lambda_d\times\lambda_d}$, so there is a natural inclusion $X_d(M)\hookrightarrow F_{G_d^n}(M,2)$.   For the strong
deformation retraction $\Omega:F(P,2)\times [0,1]\rightarrow F(P,2)$
in Lemma~\ref{defor re}, it may  be lifted naturally to a strong
deformation retraction $$\widetilde{\Omega}: F(P,2)\times
G_d^{2n}\times [0,1]\rightarrow F(P,2)\times G_d^{2n}$$ by mapping
$(a,g,t)$ to $(\Omega(a,t),g)$.
 Now assume that two points $(a,g)$ and $(a,g')$
of $F(P,2)\times G_d^{2n}$ satisfy
$(a,g)\sim_{\lambda_d\times\lambda_d}(a,g')$.

\vskip .1cm

\noindent {\bf Claim A.} $\widetilde{\Omega}(a, g,
t)=(\Omega(a,t),g)\sim_{\lambda_d\times\lambda_d} (\Omega(a,t),
g')=\widetilde{\Omega}(a, g',t)$.

\vskip .1cm

If $a\in \text{\rm int}(P^{\times 2})$, then,  by the construction
of $M$, $g=g'$. It follows that $\widetilde{\Omega}(a, g,
t)=\widetilde{\Omega}(a,g',t)$,  so
$(\Omega(a,t),g)\sim_{\lambda_d\times\lambda_d} (\Omega(a,t), g')$
regardless of whether $\Omega(a, t)$ belongs to  $\text{\rm
int}(P^{\times 2})$ or not.

If $a\in \partial(P^{\times 2})$, then $a$ belongs to $F\times P$ or $P\times F'$ where $F$ and  $F'$ are facets of $P$.
Without  loss of generality, we merely consider the case of $a\in F\times P$ in the following argument.
Let  $\sigma(t)=\Omega(a, t)$ be the path from $\Omega(a, 0)$ to $\Omega(a, 1)$. For such a path $\sigma$, we see  from the proof of Lemma~\ref{defor re} that there exists  a sequence of faces in $F\times P$
$$F\times P\supseteq Q_1\times Q'_1\supset \cdots\supset Q_{l-1}\times Q'_{l-1}\supset Q_l\times Q'_l$$
with $Q_i\cap Q'_i\not=\emptyset$ for $i=1, ..., l-1$ and $Q_l\cap
Q'_l=\emptyset$, such that $\sigma(t)$ continuously runs from
$\sigma(0)=\Omega(a,0)=a\in \text{\rm int}(Q_1\times Q'_1)$ to
$\sigma(1)=\Omega(a, 1)\in \text{\rm int}(Q_l\times Q'_l)$
   through
 $$\text{\rm int}(Q_1\times Q'_1)\supset \cdots\supset \text{\rm int}(Q_{l-1}\times Q'_{l-1})\supset \text{\rm int}(Q_l\times Q'_l).$$
    Thus, by the definition of $\sim_{\lambda_d\times \lambda_d}$, we have that $g^{-1}g'\in G_{Q_1}\times G_{Q'_1}$. On the other hand, by the construction of $M$, we have  the following sequence of subgroups of $G_d^{2n}$
$$G_{Q_1}\times G_{Q'_1}< \cdots<G_{Q_{l-1}}\times G_{Q'_{l-1}}< G_{Q_l}\times G_{Q'_l}$$
where $G_Q$ is the subgroup of $G_d^n$, determined by $Q$ and the
characteristic function of $P$ (see subsection 2.1). Note that $G_{Q_i}\times G_{Q'_i}= G_{Q_i\times Q'_i}$ for all $1\leq i\leq l$ by the definition of $\sim_{\lambda_d\times \lambda_d}$. Furthermore, we have that
$g^{-1}g'\in G_{Q_i\times Q'_i}$ for all $1\leq i\leq l$. Thus,
$(\Omega(a,t),g)\sim_{\lambda_d\times\lambda_d} (\Omega(a,t), g')$.

Moreover, we conclude by Claim A that $\widetilde{\Omega}$ descends
to an equivariant strong deformation retraction of $F_{G_d^n}(M,2)$
onto $X_d(M)$.
\end{proof}

\begin{cor}\label{equiv-coh}
 The equivariant cohomologies of $F_{G_d^n}(F, 2)$ and $X_d(M)$ are isomorphic, i.e.,
$$H^*_{G_d^{2n}}(F_{G_d^n}(F, 2))\cong H^*_{G_d^{2n}}(X_d(M)).$$
\end{cor}
\begin{proof}
Let $\Upsilon: F_{G_d^n}(F, 2)\times
[0,1]\longrightarrow F_{G_d^n}(F, 2)$ be the equivariant strong
deformation retraction of $F_{G_d^n}(F, 2)$ onto $X_d(M)$.
Consider the following equivariant lifting of  $\Upsilon$
$$\widetilde{\Upsilon}:EG_d^{2n}\times F_{G_d^n}(F,
2)\times [0,1]\longrightarrow EG_d^{2n}\times F_{G_d^n}(F, 2)$$ by
mapping $(x, y, t)$ to $(x, \Upsilon(y, t))$. This lifting
$\widetilde{\Upsilon}$ descends to a deformation retraction
$$EG_d^{2n}\times_{G_d^{2n}} F_{G_d^n}(F, 2)\times [0,1]\longrightarrow
EG_d^{2n}\times_{G_d^{2n}} F_{G_d^n}(F, 2)$$ of
$EG_d^{2n}\times_{G_d^{2n}} F_{G_d^n}(F, 2)$ onto
$EG_d^{2n}\times_{G_d^{2n}} X_d(M)$, which induces the required
result.
\end{proof}

\subsection{The simplicial complex $K_P$ associated to $X_d(M)$}

By the intersection property of all $(\pi_d^{-1})^{\times 2}(F_1\times F_2), F_1, F_2\in \mathcal{F}(P)$ with $F_1\cap F_2=\emptyset$,
     in the way as shown in section~\ref{MV},  $X_d(M)$ can determine a simplicial complex $K_P$, which is
     exactly the dual cell decomposition of $\mathcal{A}(P,2)$ as a polyhedron since the dual cell decomposition of each simple convex polytope is
a simplicial complex. $K_P$ only depends upon the combinatorial structure of $P$.
    Actually,  $K_P$ can also be directly determined by the intersection property of
    all faces $F_1\times F_2$ with $F_1\cap F_2=\emptyset$ in $P^{\times 2}$, where $ F_1, F_2\in \mathcal{F}(P)$.

    Specifically, let $\mathcal{B}(P)$ be the set consisting of all
    faces $F_1\times F_2$ with $F_1\cap F_2=\emptyset$ in $P^{\times 2}$ where $ F_1, F_2\in
    \mathcal{F}(P)$, such that each face in $\mathcal{B}(P)$ is not
    a proper face of other faces in $\mathcal{B}(P)$. Then we have
    that
    $$\mathcal{A}(P,2)=\bigcup_{B\in \mathcal{B}(P)}B.$$
Now the associated simplicial complex $K_P$ is defined as follows:
regard $\mathcal{B}(P)$ as the vertex set of $K_P$, and a simplex of
$K_P$ is a subset $\{B_1, ..., B_r\}$ of $\mathcal{B}(P)$ such that
the intersection $\cap_{i=1}^rB_i$ is nonempty.
\begin{example}\label{dual decom}
Let $P(m)$ be an $m$-polygon with facets $F_1, ..., F_m$ and
vertices $v_1, ..., v_m$, as shown in the following diagram:
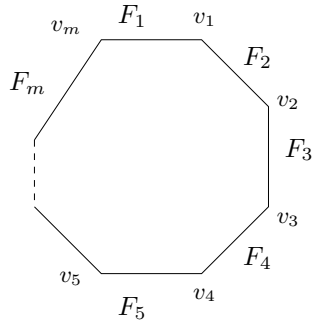
\[   \begin{picture}(0,0)%
\includegraphics{f4.pstex}%
\end{picture}%
\setlength{\unitlength}{2763sp}%
\begingroup\makeatletter\ifx\SetFigFont\undefined%
\gdef\SetFigFont#1#2#3#4#5{%
  \reset@font\fontsize{#1}{#2pt}%
  \fontfamily{#3}\fontseries{#4}\fontshape{#5}%
  \selectfont}%
\fi\endgroup%
\begin{picture}(3028,2856)(2476,-2599)
\put(2476,-511){\makebox(0,0)[lb]{\smash{{\SetFigFont{10}{12.0}{\rmdefault}{\mddefault}{\updefault}$F_m$}}}}
\put(3451, 89){\makebox(0,0)[lb]{\smash{{\SetFigFont{10}{12.0}{\rmdefault}{\mddefault}{\updefault}$F_1$}}}}
\put(4126, 89){\makebox(0,0)[lb]{\smash{{\SetFigFont{9}{10.8}{\rmdefault}{\mddefault}{\updefault}$v_1$}}}}
\put(4576,-286){\makebox(0,0)[lb]{\smash{{\SetFigFont{10}{12.0}{\rmdefault}{\mddefault}{\updefault}$F_2$}}}}
\put(4876,-661){\makebox(0,0)[lb]{\smash{{\SetFigFont{9}{10.8}{\rmdefault}{\mddefault}{\updefault}$v_2$}}}}
\put(4951,-1111){\makebox(0,0)[lb]{\smash{{\SetFigFont{10}{12.0}{\rmdefault}{\mddefault}{\updefault}$F_3$}}}}
\put(4876,-1711){\makebox(0,0)[lb]{\smash{{\SetFigFont{9}{10.8}{\rmdefault}{\mddefault}{\updefault}$v_3$}}}}
\put(4576,-2086){\makebox(0,0)[lb]{\smash{{\SetFigFont{10}{12.0}{\rmdefault}{\mddefault}{\updefault}$F_4$}}}}
\put(3451,-2536){\makebox(0,0)[lb]{\smash{{\SetFigFont{10}{12.0}{\rmdefault}{\mddefault}{\updefault}$F_5$}}}}
\put(2926,-2236){\makebox(0,0)[lb]{\smash{{\SetFigFont{9}{10.8}{\rmdefault}{\mddefault}{\updefault}$v_5$}}}}
\put(4126,-2386){\makebox(0,0)[lb]{\smash{{\SetFigFont{9}{10.8}{\rmdefault}{\mddefault}{\updefault}$v_4$}}}}
\put(2851, 14){\makebox(0,0)[lb]{\smash{{\SetFigFont{9}{10.8}{\rmdefault}{\mddefault}{\updefault}$v_m$}}}}
\end{picture}%
\centering
   \]
  Now let us look at the simplicial complex $K_{P(m)}$ dual to $\mathcal{A}(P,2)$.
   $K_{P(3)}$ is a 6-polygon with vertices
   $v_1\times F_3, v_2\times F_1, v_3\times F_2, F_1\times v_2, F_2\times v_3, F_3\times v_1$,
    $K_{P(4)}$ is a 4-polygon with four vertices $F_1\times F_3, F_2\times F_4, F_3\times F_1, F_4\times F_2$, and
     $K_{P(5)}$ is a 2-dimensional simplicial complex with 10 vertices. Actually $K_{P(5)}$ is exactly an annuls as shown in the following diagram:
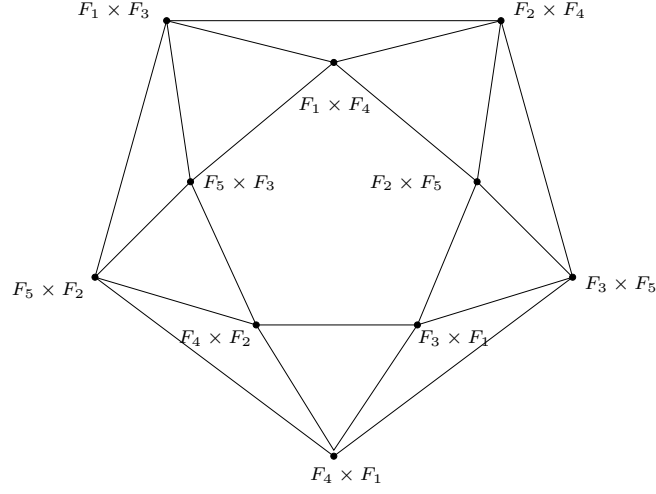
\[   \begin{picture}(0,0)%
\includegraphics{f5.pstex}%
\end{picture}%
\setlength{\unitlength}{1973sp}%
\begingroup\makeatletter\ifx\SetFigFont\undefined%
\gdef\SetFigFont#1#2#3#4#5{%
  \reset@font\fontsize{#1}{#2pt}%
  \fontfamily{#3}\fontseries{#4}\fontshape{#5}%
  \selectfont}%
\fi\endgroup%
\begin{picture}(7200,6081)(4351,-4474)
\put(7951,239){\makebox(0,0)[lb]{\smash{{\SetFigFont{7}{8.4}{\rmdefault}{\mddefault}{\updefault}$F_1\times F_4$}}}}
\put(6751,-736){\makebox(0,0)[lb]{\smash{{\SetFigFont{7}{8.4}{\rmdefault}{\mddefault}{\updefault}$F_5\times F_3$}}}}
\put(10651,1439){\makebox(0,0)[lb]{\smash{{\SetFigFont{7}{8.4}{\rmdefault}{\mddefault}{\updefault}$F_2\times F_4$}}}}
\put(11551,-2011){\makebox(0,0)[lb]{\smash{{\SetFigFont{7}{8.4}{\rmdefault}{\mddefault}{\updefault}$F_3\times F_5$}}}}
\put(8851,-736){\makebox(0,0)[lb]{\smash{{\SetFigFont{7}{8.4}{\rmdefault}{\mddefault}{\updefault}$F_2\times F_5$}}}}
\put(8101,-4411){\makebox(0,0)[lb]{\smash{{\SetFigFont{7}{8.4}{\rmdefault}{\mddefault}{\updefault}$F_4\times F_1$}}}}
\put(4351,-2086){\makebox(0,0)[lb]{\smash{{\SetFigFont{7}{8.4}{\rmdefault}{\mddefault}{\updefault}$F_5\times F_2$}}}}
\put(5176,1439){\makebox(0,0)[lb]{\smash{{\SetFigFont{7}{8.4}{\rmdefault}{\mddefault}{\updefault}$F_1\times F_3$}}}}
\put(6451,-2686){\makebox(0,0)[lb]{\smash{{\SetFigFont{7}{8.4}{\rmdefault}{\mddefault}{\updefault}$F_4\times F_2$}}}}
\put(9451,-2686){\makebox(0,0)[lb]{\smash{{\SetFigFont{7}{8.4}{\rmdefault}{\mddefault}{\updefault}$F_3\times F_1$}}}}
\end{picture}%
\centering
   \]
In general, when $m>5$, $K_{P(m)}$ is a 3-dimensional simplicial
complex with vertex set $\{F_i\times F_j| F_i\cap F_j=\emptyset\}$
such that there are 3-dimensional simplices  of the form
$$\{F_i\times F_j, F_{i+1}\times F_j, F_i\times F_{j+1}, F_{i+1}\times F_{j+1}\}$$
where $F_{i+1}$ will be $F_1$ if $i=m$, and $F_{j+1}$ will be $F_1$ if $j=m$,
and some additional 2-dimensional
simplices of the form
$$\{F_i\times F_{i+2}, F_i\times F_{i+3}, F_{i+1}\times F_{i+3}\} \text{ or }\{F_{i+2}\times F_i, F_{i+3}\times
F_i, F_{i+3}\times F_{i+1}\}$$ where $F_{i+2}$ will be $F_1, F_2$ if
$i=m-1, m$ and $F_{i+3}$ will be $F_1, F_2, F_3$  if $i=m-2, m-1,
m$. Obviously,  $K_{P(m)}$ contains $m(m-3)$ vertices. We also know
from Remark~\ref{hom-type} that $K_{P(m)}$ is homotopic to a circle.

\end{example}



\subsection{Examples for the homotopy types of $F_{G_d^n}(M,2)$}\label{exam}
Now let us look at the case $n=2$. In this case, $P$ is a polygon. Based upon Example~\ref{dual decom}, we have that

\begin{enumerate}
\item[(1)] When $P$ is a $3$-polygon, we have that $F_{{\Bbb Z}_2^2}(M,2)$ has the homotopy type of the following 1-dimensional simplicial complex
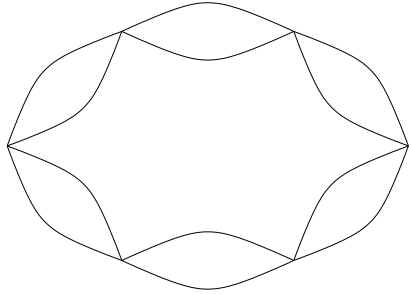
\[   \begin{picture}(0,0)%
\includegraphics{f1.pstex}%
\end{picture}%
\setlength{\unitlength}{2368sp}%
\begingroup\makeatletter\ifx\SetFigFont\undefined%
\gdef\SetFigFont#1#2#3#4#5{%
  \reset@font\fontsize{#1}{#2pt}%
  \fontfamily{#3}\fontseries{#4}\fontshape{#5}%
  \selectfont}%
\fi\endgroup%
\begin{picture}(4224,3024)(1789,-3373)
\end{picture}%
\centering
   \]
   and  $F_{T^2}(M,2)$ has the homotopy type of a 2-dimensional simplicial complex produced by replacing six circles of the above complex by six 2-spheres.
\item[(2)] When $P$ is a $4$-polygon, we have that $F_{{\Bbb Z}_2^2}(M,2)$ has the homotopy type of the following 2-dimensional simplicial complex
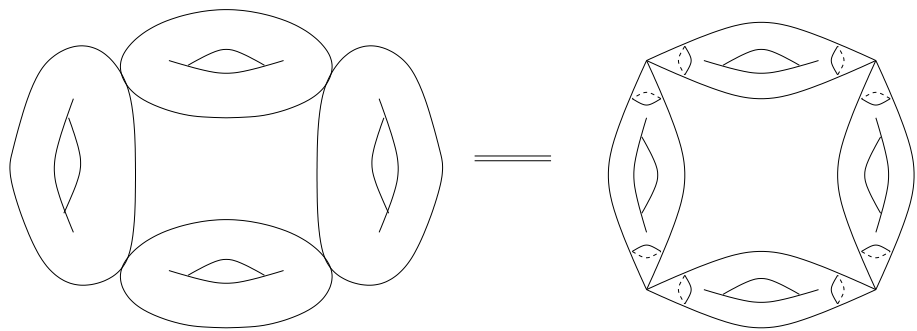
\[   \begin{picture}(0,0)%
\includegraphics{f2.pstex}%
\end{picture}%
\setlength{\unitlength}{1579sp}%
\begingroup\makeatletter\ifx\SetFigFont\undefined%
\gdef\SetFigFont#1#2#3#4#5{%
  \reset@font\fontsize{#1}{#2pt}%
  \fontfamily{#3}\fontseries{#4}\fontshape{#5}%
  \selectfont}%
\fi\endgroup%
\begin{picture}(14221,5024)(1392,-5173)
\end{picture}%
\centering
   \]
   and  $F_{T^2}(M,2)$  has the homotopy type of a 4-dimensional simplicial complex produced by replacing four tori of the above complex by four copies of $S^2\times S^2$.
 \item[(3)]  When $P$ is a $5$-polygon, we have that $F_{{\Bbb Z}_2^2}(M,2)$ has the homotopy type of
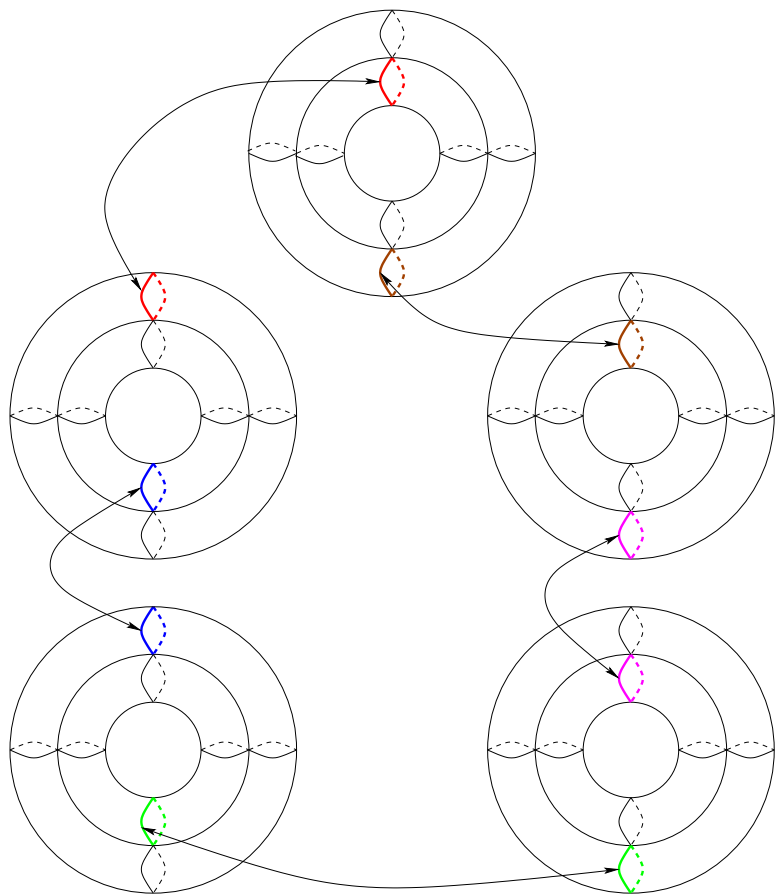
\[   \begin{picture}(0,0)%
\includegraphics{f3.pstex}%
\end{picture}%
\setlength{\unitlength}{1973sp}%
\begingroup\makeatletter\ifx\SetFigFont\undefined%
\gdef\SetFigFont#1#2#3#4#5{%
  \reset@font\fontsize{#1}{#2pt}%
  \fontfamily{#3}\fontseries{#4}\fontshape{#5}%
  \selectfont}%
\fi\endgroup%
\begin{picture}(9624,11695)(1189,-8744)
\put(2626,-8686){\makebox(0,0)[lb]{\smash{{\SetFigFont{7}{8.4}{\rmdefault}{\mddefault}{\updefault}The resulting space is obtained by gluing same colored circles together }}}}
\end{picture}%
\centering
   \]
 and  $F_{T^2}(M,2)$  has the homotopy type of a 4-dimensional simplicial complex produced by replacing all tori and circles of the above complex by $S^2\times S^2$ and $S^2$ respectively.
\end{enumerate}

\section{The homology  of $F_{G_d^n}(M,2)$ }\label{b-eq}

Throughout the following, assume that $\pi_d:M^{dn}\longrightarrow
P^n$ is a $dn$-dimensional $G_d^n$-manifold over a simple convex
polytope $P$.





\subsection{The Mayer--Vietoris spectral sequence of $F_{G_d^n}(M,2)$}
Recall that $K_P$ is the simplicial complex associated to $X_d(M)$, and it indicates the intersection property of submanifolds of
$\{(\pi_d^{-1})^{\times 2}(F_1\times F_2)| F_1, F_2\in
\mathcal{F}(P) \text{ with } F_1\cap F_2=\emptyset\}$ in $X_d(M)$.
Then we know from section~\ref{MV} that $X_d(M)$ with $K_P$
together can be associated to a Mayer--Vietoris spectral sequence
with ${\Bbb Z}$ coefficients, denoted by $E^1_{p,q}(K_P, d), ...,
E^\infty_{p,q}(K_P, d)$. Here we call it the {\em Mayer--Vietoris
spectral sequence} of $F_{G_d^n}(M,2)$.
  Then we have that (also see Theorem~\ref{spectral} in Section~\ref{MV})
 \begin{thm}\label{mv-c}
The Mayer--Vietoris spectral sequence $$\{E^1_{p,q}(K_P, d), ...,
E^\infty_{p,q}(K_P, d)\}$$ of $F_{G_d^n}(M,2)$  converges to
$H_*(F_{G_d^n}(M,2))$, i.e.,
$$H_i(F_{G_d^n}(M,2))\cong \sum_{p+q=i} E_{p,q}^\infty(K_p, d).$$
 \end{thm}

\subsection{Relation between  $E_{p,q}(K_P, 1)$ and $E_{p,q}(K_P, 2)$}
   Davis and Januszkiewicz showed in  \cite[Theorem 3.1]{dj} that the Betti numbers (mod 2 Betti numbers for $d=1$)
   of the $G_d^n$-manifold $\pi_d: M^{dn}\longrightarrow P^n$ only depend upon the combinatorics (more precisely, the $h$-vector) of
    $P^n$. Specifically, there is a height function $\phi$ on $P^n\subset {\Bbb R}^n$
    defined by $\phi(x)=<x, w>$ where $w$ is tangent to no proper face of $P^n$.
    This height function $\phi$ determines a perfect cell decomposition of $M^{dn}$ in the sense of Morse theory,
    and gives  a 1-1 correspondence between all $di$-dimensional perfect
    cells $e_v$ of $M^{dn}$ and all vertices $v$ of index $i$. The closure  of every $di$-dimensional perfect
    cell $e_v$ can be used as a generator of $H_{di}(M^{dn};R_d)$, and we denote
    this associated generator of $H_{di}(M^{dn}; R_d)$ by $\gamma_d(v,P)$.
    In particular, when $v$ runs over all vertexes of index $i$, all $\gamma_d(v,P)$
    form a basis of $H_{di}(M^{dn}; R_d)$. In addition, when $d=2$, $H_{\text{odd}}(M^{2n};{\Bbb Z})=0$
    and $H_{2i}(M^{2n};
    {\Bbb Z})$ is a free abelian group so $\dim H_{2i}(M^{2n};
    {\Bbb Z}_2)$ is equal to the rank of $H_{2i}(M^{2n};
    {\Bbb Z})$. For convenience, we still denote the generators of $H_{2i}(M^{2n};
    {\Bbb Z}_2)$ by those $\gamma_2(v,P)$. Thus there is a natural
    isomorphism
\begin{equation}\label{na-iso}
\zeta: H_i(M^n; {\Bbb Z}_2)\longrightarrow H_{2i}(M^{2n};{\Bbb
Z}_2)\end{equation}
 by mapping $\gamma_1(v,P)$ to $\gamma_2(v,P)$.

  Now let $F^l$ be a $l$-face of $P^n$ and $\eta: F^l\hookrightarrow P^n$ be the natural
imbedding. By $\pi_d^{-1}\eta: \pi_d^{-1}(F^l)\hookrightarrow
\pi_d^{-1}(P^n)=M^{dn}$ we denote the pull-back of $\eta$ via
$\pi_d^{-1}$. Note that $\pi_d^{-1}(F^l)\longrightarrow F^l$ is
still a $G_d^l$-manifold (see \cite[Lemma 1.3]{dj}), so there is
still a natural isomorphism $H_i(\pi_1^{-1}(F^l);{\Bbb
Z}_2)\longrightarrow H_{2i}(\pi_2^{-1}(F^l);{\Bbb Z}_2)$, also
denoted by $\zeta$.
\begin{lem}\label{na-hom}
Let $(\pi_d^{-1}\eta)_*: H_*(\pi_d^{-1}(F^l);{\Bbb Z}_2)\longrightarrow H_*(\pi_d^{-1}(P^n);{\Bbb Z}_2)$ be the homomorphism induced by $\pi_d^{-1}\eta$. Then
 $\zeta$ is
commutative with $(\pi_d^{-1}\eta)_*$, i.e.,
$$(\pi_2^{-1}\eta)_*\circ\zeta=\zeta\circ(\pi_1^{-1}\eta)_*.$$
\end{lem}
\begin{proof}
The restriction  to $F^l$ of each height function
 on $P^n$ gives a height function on $F^l$. We can always
choose a height function $\phi$ on $P^n$ such that the restriction
$\phi|_{F^l}$ to $F^l$ possesses the property that for each vertex
$v$ of $F^l\hookrightarrow P^n$, the index $\text{ind}(v,
\phi|_{F^l})$ at $v$ of $\phi|_{F^l}$ on $F^l$ is equal to the
$\text{ind}(v, \phi)$ at $v$ of $\phi$ on $P^n$. Then we see from
the proof of \cite[Theorem 3.1]{dj} that the perfect cell structure
of $\pi_d^{-1}(F^l)$ determined by $\phi|_{F^l}$ agrees with that of
$\pi_d^{-1}(P^n)$ determined by $\phi$. In other words, as
CW-complexes,
  $\pi_d^{-1}(F^l)$  is  a subcomplex of $\pi_d^{-1}(P^n)$. This means that
$(\pi_d^{-1}\eta)_*(\gamma_d(v,F))=\gamma_d(v,P)$. Furthermore, for
every generator $\gamma_1(v,F)$ of $H_i(\pi_1^{-1}(F^l);{\Bbb
Z}_2)$, we have
$$(\pi_2^{-1}\eta)_*\circ\zeta(\gamma_1(v,F))=(\pi_2^{-1}\eta)_*(\gamma_2(v,F))$$
$$=\gamma_2(v,P)=\zeta(\gamma_1(v,P))=\zeta\circ(\pi_1^{-1}\eta)_*(\gamma_1(v,F)).$$
This complete the proof.
\end{proof}

From the equation~(\ref{e1}) in Remark~\ref{complex}, we see that
there is the following isomorphism
\begin{equation}\label{equation}
E^1_{p,q}(K_P, d)\cong \bigoplus_{a\in K_P\atop \dim a=p}H_q(X^d_a)
\end{equation} where $X^d_a$ is the intersection of some
submanifolds of the form $(\pi_d^{-1})^{\times 2}(F_1\times F_2)$,
$F_1, F_2\in \mathcal{F}(P^n)$ with $F_1\cap F_2=\emptyset$. Since
the intersection of those faces of the form $F_1\times F_2$ is still
a face of $P^n\times P^n$, we have that $X_a^1$ is a small cover and
$X^2_a$ is a quasitoric manifold, such that their orbit polytopes
are the same. Thus, as shown in (\ref{na-iso}), there is a natural
isomorphism $\zeta_a: H_q(X_a^1;{\Bbb Z}_2)\longrightarrow
H_{2q}(X_a^2;{\Bbb Z}_2)$.  Then using all isomorphisms $\zeta_a$,
$a\in K_P$ with $\dim a=p$, we induce an isomorphism
$$E^1_{p,q}(K_P, 1)\otimes{\Bbb Z}_2\longrightarrow
E^1_{p,2q}(K_P,2)\otimes{\Bbb Z}_2$$ still denoted by $\zeta$.

On the other hand, as shown in
Remark~\ref{complex}, the differential $\text{\bf d}^{(d)}_1$
 on $E^1_{*,dq}(K_P,d)\otimes {\Bbb Z}_2$ can be explicitly described in terms of
 homomorphisms $$\ell_d: H_*(X^d_a; {\Bbb Z}_2)\longrightarrow H_*(X^d_b; {\Bbb Z}_2)$$ (induced by the  the natural
 imbedding $X^d_a\hookrightarrow X^d_b$) where $b$ is a face of $a$
 and $a\in K$. As we have seen before, $X^d_a$ and $X^d_b$ are small
 covers if $d=1$ and quasitoric manifolds if $d=2$. Thus there are
 two faces $F_a$ and $F_b$ of $P^n\times P^n$ such that $F_a$ is a
 face of $F_b$, and  the natural imbedding $F_a\hookrightarrow F_b$
 induces the natural imbedding $(\pi_d^{-1})^{\times
 2}(F_a)=X_a^d\hookrightarrow (\pi_d^{-1})^{\times 2}(F_b)=X_b^d$.
 Now we fix a height function $\Phi$ on $P^n\times P^n$ as a simple convex polytope, such that  the
 restrictions to $F_a$ and $F_b$ of this height function give the
 perfect cell decompositions of $X^d_a$ and $X^d_b$ respectively, which are
 compatible with the perfect cell decomposition of $(\pi_d^{-1})^{\times 2}(P^n\times
 P^n)=(M^{dn})^{\times 2}$ determined by $\Phi$. Namely,  every perfect cell of $X_a$ is  a perfect cell of
 $X_b$, and in particular, all perfect cells of $X^d_a$ and $X^d_b$
 are also perfect cells of $(\pi_d^{-1})^{\times 2}(P^n\times
 P^n)=(M^{dn})^{\times 2}$. Therefore, if we use $X_a^d$ and $X_b^d$ to
 replace the $\pi_d^{-1}(F^l)$ and $\pi_d^{-1}(P^n)$ of Lemma~\ref{na-hom} respectively, then we have that
 $\zeta\circ \ell_1=\ell_2\circ \zeta$. Moreover, we have that
 $$\zeta\circ \text{\bf d}^{(1)}_1=\text{\bf d}^{(2)}_1\circ \zeta$$
 so $\zeta$ is a chain map between two chain complexes $(E^1(K_P; 1)\otimes{\Bbb Z}_2, \text{\bf d}^{(1)}_1)$ and $(E^1(K_P; 2)\otimes{\Bbb Z}_2, \text{\bf d}^{(2)}_1)$. Since $\zeta$ is an isomorphism, we
 conclude that

\begin{thm}\label{betti} The chain isomorphism $\zeta$ induces the following isomorphism
$$E^2_{p,q}(K_P,1)\otimes{\Bbb Z}_2\cong E^2_{p,2q}(K_P, 2)\otimes{\Bbb Z}_2.$$
\end{thm}

\subsection{Proof of Theorem~\ref{eq-coh}}
 \cite[Theorem 3.1]{dj} tells us
that all perfect cells of each quasi-toric manifold are of even
dimension, so by (\ref{equation}) we have that  $E^1_{p,q}(K_P,2)$
vanishes for odd $q$. This means that the differential $$\text{\bf d}^{(2)}_r:
E^r_{p, q}(K_P, 2)\longrightarrow E^r_{p-r, q+r-1}(K_P, 2)$$ is a
zero-homomorphism for $r\geq 2$, so the Mayer--Vietoris spectral
sequence of $F_{T^n}(M,2)$ collapses at the $E^2$-term, and
$$H_i(F_{T^m}(M,2))\cong \bigoplus_{p+q=i}E^2_{p,q}(K_P, 2).$$
This completes the proof. \hfill$\Box$

\section{Calculation of the (mod 2) homology of $F_{G_d^2}(M,2)$ and  $F_{G_d^n}(M,2)$ for  $P$  an $n$-simplex}\label{ca}
In this section, using  Theorem~\ref{main re} and the Mayer-Vietoris spectral sequence we  calculate the (mod 2) homology  of $F_{G_d^2}(M,2)$ and  $F_{G_d^n}(M,2)$ for $P$  an  $n$-simplex.
Our results are stated as follows:

\begin{prop}\label{b1}
Let $\pi_d: M\longrightarrow P(m)$ be a $2d$-dimensional
$G_d^2$-manifold over an $m$-polygon $P(m)$. When $m=3$, all nonzero
Betti numbers of $F_{{\Bbb Z}_2^2}(M,2)$ $($resp. $F_{T^2}(M, 2))$
are $$(b_0, b_1)=(1, 7) \ \ \text{\rm (resp.} \ (b_0,
 b_1, b_2)=(1, 1, 6));$$ when $m>3$, all nonzero Betti numbers of
$F_{{\Bbb Z}_2^2}(M,2)$ $($resp. $F_{T^2}(M, 2))$ are
$$(b_0, b_1,
b_2)=(1, 2m+1, m(m-3))\ \ \text{\rm (resp.} \ (b_0, b_1, b_2,
b_4)=(1,1, 2m, m(m-3))).$$ In particular, the non-vanishing homology
of $F_{{\Bbb Z}_2^2}(M,2)$ is free abelian.
\end{prop}

\begin{rem}
As we have seen in Proposition~\ref{b1}, we actually determine the
integral homology of $F_{{\Bbb Z}_2^2}(M,2)$. However, unlike
2-dimensional small covers, the non-vanishing homology of $F_{{\Bbb
Z}_2^2}(M,2)$ has no torsion. In addition, unlike quasitoric
manifolds, odd-dimensional homology of $F_{{T}^n}(M,2)$ may be
non-vanishing in general. This can also be seen from the following
proposition.
\end{rem}

\begin{prop}\label{b2}
Let $\pi_d: M\longrightarrow \Delta^n$ be a $dn$-dimensional
$G_d^n$-manifold over an $n$-simplex $\Delta^n$. Then $M$ is one of ${\Bbb R}P^n, {\Bbb C}P^n$ or $\overline{{\Bbb C}P}^n$ $($see \cite[Page 426]{bp}$)$.   When $d=1$, all
nonzero mod $2$ Betti numbers of $F_{{\Bbb Z}_2^n}({\Bbb R}P^n,2)$ are
$$(b^{{\Bbb Z}_2}_0, b^{{\Bbb Z}_2}_1, ..., b^{{\Bbb Z}_2}_{n-2},
b^{{\Bbb Z}_2}_{n-1})=(1, 2, ..., n-1, {{3^{n+1}+2n-3}\over 4}).$$
When $d=2$, all nonzero Betti numbers of $F_{T^n}({\Bbb C}P^n,2)$
are
$$(b_0,b_1, ..., b_k, ..., b_{2n-2})$$
where $b_k=\begin{cases} {k\over 2}+1+f_{k-n+1} & \text{ if } 2|k \\
f_{k-n+1} & \text{ if } 2\nmid k\end{cases}$  and
$f_i=\sum_{s=0}^i{{n+1}\choose{s}}{{n-s-1}\choose{i-s}}.$
\end{prop}

\subsection{The integral homology of $F_{G_d^2}(M,2)$}
Let $\pi_d: M\longrightarrow P(m)$ be a $2d$-dimensional
$G_d^2$-manifold over an $m$-polygon $P(m)$.
  Theorem~\ref{main re} tells us that $F_{G_d^2}(M,2)$ is homotopic to
$$X_d(P(m))=\begin{cases}
\bigcup_{F_i\cap F_j=\emptyset}(\pi_d^{-1})^{\times 2}(F_i\times F_j) & \text{ if } m>3\\
\bigcup_{v_i\cap F_j=\emptyset}(\pi_d^{-1})^{\times 2}(v_i\times
F_j)\bigcup\bigcup_{F_i\cap v_j=\emptyset}(\pi_d^{-1})^{\times
2}(F_i\times v_j) & \text{ if } m=3.
\end{cases}$$
In Example~\ref{dual decom}, we have given an analysis on the
structure of $K_{P(m)}$.
 For the convenience of calculation, as shown in Section~\ref{MV}, we can choose a
 locally nice subcomplex $L_{P(m)}$ of $K_{P(m)}$ in the use of the Mayer-Vietoris spectral sequence of $F_{G_d^2}(M,2)$ (for the notion of a locally nice subcomplex, see Defintion~\ref{local}).

  When $m\leq 5$, take $L_{P(m)}=K_{P(m)}$. When $m\geq 6$, we take $L_{P(m)}$ in such a way that $L_{P(m)}$ contains $\emptyset$ and all vertices of $K_{P(m)}$, and
$2m(m-4)$ 2-dimensional simplices of the following forms
$$\{F_i\times F_j, F_{i+1}\times F_j, F_{i+1}\times F_{j+1}\} \text{ and } \{F_i\times F_j, F_i\times F_{j+1},
F_{i+1}\times F_{j+1}\}.$$ In this case, we may check that
$L_{P(m)}$ is exactly an annulus, and it has $m(3m-11)$ 1-simplices.
For example, when $m=6$, $L_{P(6)}$ is a 2-dimensional simplicial
complex as shown in the following picture:
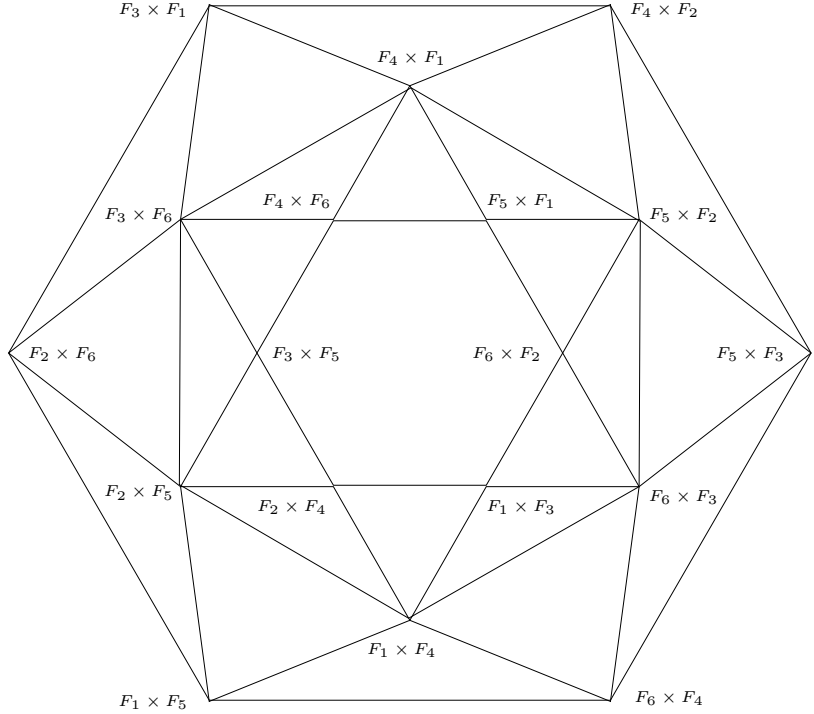
\[   \begin{picture}(0,0)%
\includegraphics{f6.pstex}%
\end{picture}%
\setlength{\unitlength}{1579sp}%
\begingroup\makeatletter\ifx\SetFigFont\undefined%
\gdef\SetFigFont#1#2#3#4#5{%
  \reset@font\fontsize{#1}{#2pt}%
  \fontfamily{#3}\fontseries{#4}\fontshape{#5}%
  \selectfont}%
\fi\endgroup%
\begin{picture}(12624,11106)(-311,-6274)
\put(1426,4664){\makebox(0,0)[lb]{\smash{{\SetFigFont{6}{7.2}{\rmdefault}{\mddefault}{\updefault}$F_3\times F_1$}}}}
\put(7201,-3136){\makebox(0,0)[lb]{\smash{{\SetFigFont{6}{7.2}{\rmdefault}{\mddefault}{\updefault}$F_1\times F_3$}}}}
\put(9751,1439){\makebox(0,0)[lb]{\smash{{\SetFigFont{6}{7.2}{\rmdefault}{\mddefault}{\updefault}$F_5\times F_2$}}}}
\put(9751,-2986){\makebox(0,0)[lb]{\smash{{\SetFigFont{6}{7.2}{\rmdefault}{\mddefault}{\updefault}$F_6\times F_3$}}}}
\put(9526,-6136){\makebox(0,0)[lb]{\smash{{\SetFigFont{6}{7.2}{\rmdefault}{\mddefault}{\updefault}$F_6\times F_4$}}}}
\put(3826,-736){\makebox(0,0)[lb]{\smash{{\SetFigFont{6}{7.2}{\rmdefault}{\mddefault}{\updefault}$F_3\times F_5$}}}}
\put(7201,1664){\makebox(0,0)[lb]{\smash{{\SetFigFont{6}{7.2}{\rmdefault}{\mddefault}{\updefault}$F_5\times F_1$}}}}
\put(5326,-5386){\makebox(0,0)[lb]{\smash{{\SetFigFont{6}{7.2}{\rmdefault}{\mddefault}{\updefault}$F_1\times F_4$}}}}
\put(  1,-736){\makebox(0,0)[lb]{\smash{{\SetFigFont{6}{7.2}{\rmdefault}{\mddefault}{\updefault}$F_2\times F_6$}}}}
\put(9451,4664){\makebox(0,0)[lb]{\smash{{\SetFigFont{6}{7.2}{\rmdefault}{\mddefault}{\updefault}$F_4\times F_2$}}}}
\put(1201,-2911){\makebox(0,0)[lb]{\smash{{\SetFigFont{6}{7.2}{\rmdefault}{\mddefault}{\updefault}$F_2\times F_5$}}}}
\put(6976,-736){\makebox(0,0)[lb]{\smash{{\SetFigFont{6}{7.2}{\rmdefault}{\mddefault}{\updefault}$F_6\times F_2$}}}}
\put(10801,-736){\makebox(0,0)[lb]{\smash{{\SetFigFont{6}{7.2}{\rmdefault}{\mddefault}{\updefault}$F_5\times F_3$}}}}
\put(5476,3914){\makebox(0,0)[lb]{\smash{{\SetFigFont{6}{7.2}{\rmdefault}{\mddefault}{\updefault}$F_4\times F_1$}}}}
\put(3676,1664){\makebox(0,0)[lb]{\smash{{\SetFigFont{6}{7.2}{\rmdefault}{\mddefault}{\updefault}$F_4\times F_6$}}}}
\put(3601,-3136){\makebox(0,0)[lb]{\smash{{\SetFigFont{6}{7.2}{\rmdefault}{\mddefault}{\updefault}$F_2\times F_4$}}}}
\put(1426,-6211){\makebox(0,0)[lb]{\smash{{\SetFigFont{6}{7.2}{\rmdefault}{\mddefault}{\updefault}$F_1\times F_5$}}}}
\put(1201,1439){\makebox(0,0)[lb]{\smash{{\SetFigFont{6}{7.2}{\rmdefault}{\mddefault}{\updefault}$F_3\times F_6$}}}}
\end{picture}%
\centering
   \]
   With the above arguments together, we have

   \begin{lem} \label{annulus}
   When $m\leq 4$, $L_{P(3)}$ is a 6-polygon and $L_{P(4)}$ is a 4-polygon. When $m\geq 5$, $L_{P(m)}$ is a triangulation of an annulus  with $m(m-3)$ vertices, $m(3m-11)$ 1-simplices and $2m(m-4)$ 2-simplices.
   \end{lem}
In the following discussion, by $E^1_{p,q}(L_{P(m)}, d), ..., E_{p,q}^\infty(L_{P(m)}, d)$ we denote the Mayer-Vietoris spectral sequence determined by $X_d(P(m))$ with $L_{P(m)}$ together.

Now, to complete the proof of Proposition~\ref{b1}, it suffices to show the following result.

\begin{prop}\label{dim 2}
$X_d(P(3))$ is a d-dimensional connected CW complex whose nonzero Betti numbers are
$(b_0, b_1)=(1, 7)$ if $d=1$ and $(b_0,
 b_1, b_2)=(1, 1, 6)$ if $d=2$. When $m\geq 4$, $X_d(P(m))$ is a 2d-dimensional
connected CW complex whose nonzero Betti numbers are $(b_0, b_1, b_2)=(1, 2m+1,
m(m-3))$ if $d=1$ and  $(b_0, b_1, b_2, b_4)=(1,1, 2m, m(m-3)))$ if $d=2$.
\end{prop}

\begin{proof} Each vertex of $L_{P(m)}$ is of the form $F_i\times F_j$ if $m>3$, and of the form $v_i\times F_j$ or $F_i\times v_j$ if $m=3$. Since each $\pi_d^{-1}(F_i)$ is $S^1$ if $d=1$ and  $S^2$ if $d=2$,  and since $\pi_d^{-1}(v_i)$ is a point, we have that $(\pi_d^{-1})^{\times 2}(F_i\times F_j)$ is $S^1\times S^1$ if $d=1$ and  $S^2\times S^2$ if $d=2$, and $(\pi_d^{-1})^{\times 2}(v_i\times F_j)$ (or $(\pi_d^{-1})^{\times 2}(F_i\times v_j)$) is also $S^1$ if $d=1$ and  $S^2$ if $d=2$. Also, for all $m\geq 3$,  $L_{P(m)}$ is connected.
Thus, $X_d(P(3))$ is a d-dimensional connected CW complex and when $m\geq
4$, $X_d(P(m))$ is a 2d-dimensional connected CW complex.

Now we first have   that $E_{p,0}^2(L_{P(m)}, d)=H_p(S^1)$ by Remark~\ref{complex} and Lemma~\ref{annulus}, so
$E_{p,0}^2(L_{P(m)},d)=0$ if $p>1$ and $E^2_{0,0}(L_{P(m)},d)\cong E^2_{1,0}(L_{P(m)},d)\cong {\Bbb Z}$.

If $m=3$, by a direct calculation we have that  $E^1_{p,dq}(L_{P(3)},d)=0$ for $p>0$
and $q>1$, and $E^1_{0,d}(L_{P(3)},d)\cong {\Bbb Z}^6$. Thus  we have that
$$E^2_{p,dq}(L_{P(3)},d)=
\begin{cases}
0 & \text{if either $p>1$ and $q=0$ or $p>0$ and $q>1$}\\
{\Bbb Z} & \text{if $p\leq 1$ and $q=0$}\\
{\Bbb Z}^6 & \text{if $p=0$ and $q=1$}
\end{cases}
$$
so $E_{p,dq}^\infty(L_{P(3)},d)\cong E_{p,dq}^2(L_{P(3)},d)$.

 If $m>3$, then we have that $E^1_{p, dq}(L_{P(m)},d)=0$ for either $p>1$ and $q>2$ or $p>0$ and $q=2$.
 By a direct calculation, we obtain that $E^2_{0,2d}(L_{P(m)},d)\cong {\Bbb Z}^{m(m-3)}$,
 $E^1_{0,d}(L_{P(m)},d)\cong {\Bbb Z}^{2m(m-3)}$ and $E^1_{1,d}(L_{P(m)},d)\cong {\Bbb Z}^{2m(m-4)}$. Note that   ${\Bbb Z}^{2m(m-4)}$ means the trivial group 0 when $m=4$,  so $E^2_{1,d}(L_{P(4)},d)=0$ and $E^2_{0,d}(L_{P(4)},d)\cong {\Bbb Z}^8$.
 When $m>4$, we can obtain that
 $$0\longrightarrow E^1_{1,d}(L_{P(m)},d)\longrightarrow E^1_{0,d}(L_{P(m)},d)\longrightarrow 0$$
 is isomorphic to the chain complex of the simplicial complex given by the disjoint union of $2m$ copies of a segment, so
 $E^2_{1,d}(L_{P(m)},d)=0$ and $E^2_{0,d}(L_{P(m)},d)\cong {\Bbb Z}^{2m}$. Thus, if $m>3$,
 $$E^2_{p,dq}(L_{P(m)},d)=
\begin{cases}
0 & \text{if either $p>1$ and $q=0$ or $p>0$ and $q>2$}\\
{\Bbb Z} & \text{if $p\leq 1$ and $q=0$}\\
{\Bbb Z}^{2m} & \text{if $p=0$ and $q=1$}\\
{\Bbb Z}^{m(m-3)} & \text{if $p=0$ and $q=2$}
\end{cases}
$$
which implies that
$E_{p,dq}^\infty(L_{P(m)},d)\cong E_{p,dq}^2(L_{P(m)},d)$.

 Finally,  the desired Betti numbers for $X_d(P(m))$ can be read out from the expression of $E^2_{p,dq}(L_{P(m)},d)$.
\end{proof}

\begin{rem}
We actually calculate the integral homology of $F_{({\Bbb
Z}_2)^2}(M^2,2)$, and in particular, the  Mayer--Vietoris spectral
sequence of $F_{({\Bbb Z}_2)^2}(M^2,2)$ collapses at the $E^2$-term.
\end{rem}

\subsection{The case in which $P$ is an $n$-simplex $\Delta^n$ with $n>1$}
Let $\text{sd}(\text{Bd}(\Delta^{n}))$ be the barycentric subdivision of the boundary complex of $\Delta^n$, which is an $(n-1)$-dimensional simplicial complex. Each simplex of $\text{sd}(\text{Bd}(\Delta^{n}))$ will be expressed as the form
$\sigma_1\subset \sigma_2
\subset\cdots\subset \sigma_l$ where $\sigma_i$ is a face of $\text{Bd}(\Delta^{n})$ (so $0\leq \dim \sigma_i\leq n-1$), and is understood as a vertex of $\text{sd}(\text{Bd}(\Delta^{n}))$.
Let $i, j$ be non-negative integers with $
 i+j+1\leq n$. By $K_{i,j}^n$ we denote the  subcomplex of $\text{sd}(\text{Bd}(\Delta^{n}))$, formed by those simplices  $\{\sigma_1\subset \sigma_2
\subset\cdots\subset \sigma_l\mid \dim\sigma_1\geq
i,\dim\sigma_l\leq n-j-1\}$. Then we known easily that $K_{i,j}^n$
has the following properties:
 \begin{enumerate}
\item[$\bullet$]
$K_{i,j}^n$ is $(n-i-j-1)$-dimensional.
\item[$\bullet$] $K_{i,j}^n$ is connected if $n>i+j+1$.
\item[$\bullet$] For $i'\leq i$ and $j'\leq j$, $K^n_{i,j}\subseteq K^n_{i',j'}$.
\item[$\bullet$] $K_{0,j}^n$ is the barycentric subdivision of the $(n-j-1)$-dimensional skeleton of $\text{Bd}(\Delta^n)$. In particular, $K_{0,0}^n=\text{Bd}(\Delta^n)$.
\end{enumerate}

\begin{lem}\label{p1}
 $K_{i,j}^n$ is combinatorially equivalent to $K_{j,i}^n$.
\end{lem}

\begin{proof}
This follows by mapping each vertex $\sigma$
of $K_{i,j}^n$ to the vertex
$\overline{\sigma}$ of $K_{j,i}^n$ and each simplex $\sigma_1 \subset \sigma_2
\subset\cdots\subset \sigma_l$ to $\overline{\sigma_l} \subset\cdots
\subset \overline{\sigma_2} \subset \overline{\sigma_1}$, where the
$\overline{\sigma}$ means the complement of $\sigma$ in the boundary complex
$\text{Bd}(\Delta^n)$ of $\Delta^n$, i.e., $\overline{\sigma}$ is the face of $\Delta^n$, determined by those vertices which are not contained in $\sigma$.
\end{proof}

\begin{prop} \label{hom}
The homology groups $H_r(K_{i,j}^n)$  of $K_{i,j}^n$ are free
abelian. When $n=i+j+1$,
$$b_r(K_{i,j}^n)=\begin{cases}0 & \text{ if }r\neq 0 \\{{n+1}\choose{i+1}}& \text{ if }r=0. \end{cases} $$
When $n>i+j+1$,
 $$b_r(K_{i,j}^n)=\begin{cases}\sum_{s=0}^j(-1)^{s+j}{{n+1}\choose{s}}{{n-s}\choose{n-i-s}} &\text{ if }
  r=n-i-j-1\\ 1 &\text{ if } r=0\\ 0 &\text{ otherwise } \end{cases}$$
where $b_r(K_{i,j}^n)$ is the $r$-th Betti number of $K_{i,j}^n$.
\end{prop}
\begin{proof}
When $n=i+j+1$, $K_{i,j}^n$ is a 0-dimensional complex with
${{n+1}\choose{i+1}}$  simplices of dimension $i$ in
$\text{Bd}(\Delta^{n})$ as its vertices, so $b_r(K_{i,j}^n)=0$ if
$r\neq 0$ and $b_0(K_{i,j}^n)={{n+1}\choose{i+1}}$. Of course, in
this case, the homology group of $K_{i,j}^n$ has no torsion.

Now suppose that $n>i+j+1$. Given a simplex $\sigma_1 \subset
\sigma_2 \subset\cdots\subset \sigma_l$ in $K_{i,j}^n$, if
$\dim\sigma_l<n-j-1$, then it belongs to $K_{i,j+1}^n$. If
$\dim\sigma_l=n-j-1$, then  this simplex $\sigma_1 \subset \sigma_2
\subset\cdots\subset \sigma_l$ belongs to the following subcomplex
of $K_{i,j}^n$
$$\bigcup_{\sigma_r\in K_{i,j}^n
\atop \dim\sigma_r=n-j-1}\overline{\text{St}(\sigma_r, K_{i,j}^n)}$$
so \begin{equation}\label{e11} K_{i,j}^n=K_{i,j+1}^n\bigcup
\bigcup_{\sigma_r\in K_{i,j}^n \atop
\dim\sigma_r=n-j-1}\overline{\text{St}(\sigma_r, K_{i,j}^n)}.
 \end{equation}
 For each $\sigma_r\in K_{i,j}^n$ with $\dim\sigma_r=n-j-1$, we have that $$\text{Lk}(\sigma_r, K_{i,j}^n)=\{\sigma_1 \subset
\sigma_2 \subset\cdots\subset \sigma_s\mid \dim\sigma_1\geq i,
\sigma_s\subsetneq \sigma_r\}$$ and obviously it is isomorphic to
the subcomplex $K_{i,0}^{n-j-1}$ of
$\text{sd}(\text{Bd}(\sigma_l))$. Thus, we conclude that
 \begin{equation}\label{e2}
 \overline{\text{St}(\sigma_r, K_{i,j}^n)}\bigcap K_{i,j+1}^n =\text{Lk}(\sigma_r, K_{i,j}^n)= K_{i,0}^{n-j-1}.
 \end{equation}

\noindent {\bf Claim B. } {\em For $n>i+j+1$ and each integer $s$, the relative homology
group  $H_s(K_{i,j}^n,K_{i,j+1}^n)$ is a free abelian, and its Betti number
$$b_s(K_{i,j}^n,K_{i,j+1}^n)=\begin{cases}
{{n+1}\choose{j+1}}{{n-j-1}\choose{n-i-j-1}} &\text{ if } s=n-i-j-1
\\0 & \text{ otherwise. }\end{cases}$$
 }

Using the axiom of excision and (\ref{e11})--(\ref{e2}), we see that
$$H_s(K_{i,j}^n,K_{i,j+1}^n)\cong \bigoplus_{\sigma_r\in K_{i,j}^n
\atop \dim\sigma_r=n-j-1} H_s(\overline{\text{St}(\sigma_r,
K_{i,j}^n)},\text{Lk}(\sigma_r, K_{i,j}^n)).$$ For $s\geq 1$, we
have
$$H_s(\overline{\text{St}(\sigma_r, K_{i,j}^n)},\text{Lk}(\sigma_r,
K_{i,j}^n))\cong H_{s-1}(\text{Lk}(\sigma_r, K_{i,j}^n)) =
H_{s-1}(K_{i,0}^{n-j-1}).$$ Since $K_{i,0}^{n-j-1}$ has the same
homology as the $(n-i-j-2)$-skeleton of the boundary complex
$\text{Bd}(\sigma_r)$ of $\sigma_r$, we have that the reduced
homology group $\widetilde{H}_{s-1}(K_{i,0}^{n-j-1})$ vanishes if
$s-1<n-i-j-2$ and the top Betti number
$b_{n-i-j-2}(K_{i,0}^{n-j-1})={{n-j-1}\choose{n-i-j-1}}$, so Claim B
follows from this.

We note that $K^n_{0,0}$ is just $\text{sd}(\text{Bd}(\Delta^n))$,
so $H_s(K^n_{0,0})=0$ if $s\neq 0,  n-1$.  Consider the long exact
sequence
$$\cdots \rightarrow \widetilde{H}_s(K_{i,j+1}^n)\rightarrow \widetilde{H}_s(K_{i,j}^n)\rightarrow
H_s(K_{i,j}^n,K_{i,j+1}^n)\rightarrow
\widetilde{H}_{s-1}(K_{i,j+1}^n)\rightarrow\cdots.$$ Moreover, using
an induction on $i,j$ and Claim B, we may easily obtain the required
result.
\end{proof}

Now let $\pi_d: M^{dn}\longrightarrow \Delta^n$ be the
$G_d^n$-manifold over an $n$-dimensional simplex $\Delta^n$. Then we
know from Theorem~\ref{main re} that $F_{G_d^n}(M^{dn},2)$ is homotopic
to
$$X_d(M^{dn})=\bigcup_{\sigma_i\in
\mathcal{F}(\Delta^n)}(\pi_d^{-1})^{\times 2}(\sigma_i\times
\overline{\sigma_i}).$$ Obviously, $X_d(M^{dn})$ is a
$d(n-1)$-dimensional CW complex. In order to apply the theory of
Mayer-Vietoris spectral sequence developed in Section~\ref{MV}, we
choose a locally nice complex $L_{\Delta^n}$ in such a way that the
vertex set of $L_{\Delta^n}$ consists of all $(\pi_d^{-1})^{\times
2}(\sigma_i\times \overline{\sigma_i})$ where $\sigma_i\in
\mathcal{F}(\Delta^n)$, and each oriented simplex of $L_{\Delta^n}$
is of the form
$$[(\pi_d^{-1})^{\times
2}(\sigma_{i_1}\times\overline{\sigma_{i_1}}),\cdots,(\pi_d^{-1})^{\times
2}(\sigma_{i_r}\times \overline{\sigma_{i_r}})]$$ with
$\sigma_{i_1}\subset\cdots\subset\sigma_{i_r}$.  If we map
$(\pi_d^{-1})^{\times 2}(\sigma_i\times \overline{\sigma_i})$ to
$\sigma_i$, we see that $L_{\Delta^n}$ is combinatorially equivalent
to $\text{sd}(\text{Bd}(\Delta^n))$. With this understood,  we will
identify
 $L_{\Delta^n}$ with $\text{sd}(\text{Bd}(\Delta^n))$.


\begin{lem}\label{iso-k}
$$E^2_{p, dq}(L_{\Delta^n},d)\otimes R_d\cong \bigoplus_{i+j=q}H_p(K_{i,j}^n; R_d)$$
where $$R_d=\begin{cases} {\Bbb Z}_2 & \text{ if } d=1\\
{\Bbb Z} & \text{ if } d=2. \end{cases}$$
\end{lem}

\begin{proof}
Given a simplex $a$ of the form
$\sigma_{i_1}\subset\cdots\subset\sigma_{i_r}$ in $L_{\Delta^n}$, we
see that
 $$X_a^d(L_{\Delta^n})=(\pi_d^{-1})^{\times 2}(\sigma_{i_1}\times
\overline{\sigma_{i_1}})\cap\cdots\cap(\pi_d^{-1})^{\times
2}(\sigma_{i_r}\times \overline{\sigma_{i_r}})=(\pi_d^{-1})^{\times
2}(\sigma_{i_1}\times\overline{\sigma_{i_r}}).$$  We have known that
$M^{dn}$ is homeomorphic to ${\Bbb R}P^n$ if $d=1$ and ${\Bbb C}P^n$
if $d=2$,  and for each face $\sigma$ of $\Delta^n$,
$\pi_d^{-1}(\sigma)$ is homeomorphic to ${\Bbb R}P^{\dim \sigma}$ if
$d=1$ and  ${\Bbb C}P^{\dim \sigma}$ if $d=2$. Thus,
$X_a^d(L_{\Delta^n})$
 is  homeomorphic to  ${\Bbb R}P^s\times {\Bbb R}P^t$ if $d=1$ and ${\Bbb C}P^s\times {\Bbb C}P^t$
 if $d=2$ where $s=\dim \sigma_{i_1}$ and $t=\dim \overline{\sigma_{i_r}}$.
 Moreover, if $d=1$ then $$H_q(X_a^1(L_{\Delta^n});{\Bbb Z}_2)\cong H_l({\Bbb R}P^s\times
{\Bbb R}P^t; {\Bbb Z}_2)=\sum_{i+j=q}H_i({\Bbb R}P^s;{\Bbb
Z}_2)\otimes H_j({\Bbb R}P^t;{\Bbb Z}_2)$$ and if $d=2$ then
$$ H_{2q}(X_a^2(L_{\Delta^n});{\Bbb Z})\cong H_{2l}({\Bbb C}P^s\times
{\Bbb C}P^t; {\Bbb Z})=\sum_{i+j=q}H_{2i}({\Bbb C}P^s;{\Bbb
Z})\otimes H_{2j}({\Bbb C}P^t;{\Bbb Z}).$$ So
$H_{dq}(X_a^d(L_{\Delta^n}); R_d)$ is generated by
$\beta^i\otimes\gamma^j$ with $i+j=q$, where
  $\beta^i$ and
$\gamma^j$ are generators of $H_i({\Bbb R}P^s;{\Bbb Z}_2)$ and
$H_j({\Bbb R}P^t;{\Bbb Z}_2)$ respectively if $d=1$, and generators
of $H_{2i}({\Bbb C}P^s;{\Bbb Z}))$ and  $H_{2j}({\Bbb C}P^t;{\Bbb
Z}))$ respectively if $d=2$. To emphasize $\beta^i\otimes\gamma^j$
as an element of $H_{dq}(X_a^d(L_{\Delta^n}); R_d)$,  we shall
denote it by $(\beta^i\otimes\gamma^j)_a$. Define
 $$f_d:E^1_{p,dq}(L_{\Delta^n}, d)\otimes R_d=\bigoplus_{a\in
L_{\Delta^n}\\
\atop \dim a=p}H_{dq}(X_a^d(L_{\Delta^n});R_d)\longrightarrow
\bigoplus_{i+j=q}\mathcal{C}_p(K_{i,j}^n; R_d)$$
 by mapping $(\beta^i\otimes\gamma^j)_a$ to $a\in K_{i,j}^n$,
where $\mathcal{C}_*(K_{i,j}^n;R_d)$ is the chain complex of
$K_{i,j}^n$ with $R_d$ coefficients. Then $f_d$ is a chain map since
the boundary operator on $E^1_{*,dq}(L_{\Delta^n},d)\otimes R_d$
agrees with the boundary operator on
$\bigoplus_{i+j=q}\mathcal{C}_*(K_{i,j}^n; R_d)$ by
Remark~\ref{complex} in Section~\ref{MV}. We see easily that $f_d$
is a bijection, so $f_d$ is actually a chain isomorphism. Thus,
$f_d$ induces the required isomorphism in Lemma~\ref{iso-k}.
\end{proof}

\begin{rem}\label{add}
If $q=0$, then $K^n_{0,0}=\text{sd}(\text{Bd}(\Delta^n))=L_{\Delta^n}$, so
$E^2_{p, 0}(L_{\Delta^n},d)\otimes R_d\cong H_p(S^{n-1}; R_d)$. This
is also shown in Remark~\ref{complex} of Section~\ref{MV} since each
$X^d_a$ is connected.
\end{rem}

Combining Proposition~\ref{hom}, Lemma~\ref{iso-k} and
Remark~\ref{add}, we have

\begin{cor}\label{e2-hom}
$$E_{p,dq}^2(L_{\Delta^n}, d)\otimes R_d\cong \begin{cases}
0 & \text{if $p+q\neq n-1$ and $p\neq0$}\\
R_d^{q+1} & \text{if $p=0$ and $q< n-1$}\\
R_d^{2^{n+1}-2} & \text {if $p=0$ and $q=n-1$}\\
R_d^{\sum_{i+j=q}\sum_{s=0}^j(-1)^{s+j}{{n+1}\choose{s}}{{n-s}\choose{n-i-s}}}&
\text{if $p+q=n-1$ and $p\not=0$.}
\end{cases}
$$
\end{cor}

\begin{prop}\label{sm-sq}
The spectral sequence $E_{p,q}^1(L_{\Delta^n}, 1)\otimes {\Bbb Z}_2,  ...,
E_{p,q}^\infty(L_{\Delta^n}, 1)\otimes {\Bbb Z}_2$ collapses at the $E_{p,q}^2$-term. Furthermore,
$$ H_i(F_{({\Bbb Z}_2)^n}(M^n,2); {\Bbb Z}_2)\cong\bigoplus_{p+q=i} E_{p,q}^2(L_{\Delta^n}, 1)\otimes {\Bbb Z}_2.$$
\end{prop}

\begin{proof}
We see by Corollary~\ref{e2-hom} that for $r\geq 2$, the differential
$${\bf d}_r: E_{p,q}^r(L_{\Delta^n}, 1)\otimes {\Bbb Z}_2\longrightarrow E_{p-r,q+r-1}^r(L_{\Delta^n}, 1)\otimes {\Bbb Z}_2$$
is a zero homomorphism if $p\not=r$ or $q+r\not =n-1$. This also means that
 $E_{p,q}^\infty(L_{\Delta^n}, 1)\otimes {\Bbb Z}_2\cong  E_{p,q}^2(L_{\Delta^n}, 1)\otimes {\Bbb Z}_2$ if either $p+q<n-2$ or $p+q>n-1$.
So, to complete the proof of Proposition~\ref{sm-sq}, it suffices to show that if $p=r$ and $q+r=n-1$, then
 \begin{equation}\label{diff}
 {\bf d}_r: E_{r,n-r-1}^r(L_{\Delta^n}, 1)\otimes {\Bbb Z}_2\longrightarrow E_{0,n-2}^r(L_{\Delta^n}, 1)\otimes {\Bbb Z}_2
 \end{equation}
 is also a zero homomorphism. Actually, this is also equivalent to
 proving that $H_{n-2}(X_1(M^n);{\Bbb Z}_2)\cong E_{0,n-2}^\infty(L_{\Delta^n}, 1)\otimes {\Bbb Z}_2\cong E_{0,n-2}^2(L_{\Delta^n}, 1)\otimes {\Bbb
 Z}_2$, which is isomorphic to ${\Bbb Z}_2^{n-1}$ by Corollary~\ref{e2-hom}.

 \vskip .2cm
\noindent {\bf Claim C.} $H_{n-2}(X_1(M^n);{\Bbb Z}_2)\cong {\Bbb
Z}_2^{n-1}$. \vskip .2cm

Regard $\Delta^n$ as a facet of an $(n+1)$-dimensional simplex
$\Delta^{n+1}$, and let $\pi_1: M^{n+1}\longrightarrow \Delta^{n+1}$
be the small cover over $\Delta^{n+1}$. Then
$M^n=\pi_1^{-1}(\Delta^n)$. Of course, we also have that
$E_{p,q}^\infty(L_{\Delta^{n+1}}, 1)\otimes {\Bbb Z}_2\cong
E_{p,q}^2(L_{\Delta^{n+1}}, 1)\otimes {\Bbb Z}_2$ if $p+q<n-1$, so
using Theorem~\ref{mv-c}, we also have that
$$ H_i(X_1(M^{n+1});{\Bbb Z}_2)\cong E_{0,i}^2(L_{\Delta^{n+1}}, 1)\otimes {\Bbb Z}_2\cong {\Bbb Z}_2^{i+1}$$ for $i<n-1$.
 In particular, $H_{n-2}(X_1(M^{n+1});{\Bbb Z}_2)\cong E_{0,n-2}^2(L_{\Delta^{n+1}}, 1)\otimes {\Bbb Z}_2\cong {\Bbb Z}_2^{n-1}$.
From the proof of Lemma~\ref{iso-k}, we have the following
isomorphism
$$E_{0,n-2}^1(L_{\Delta^{n+1}}, 1)\otimes {\Bbb Z}_2=\bigoplus_{a\in L_{\Delta^{n+1}}\\
\atop \dim a=0}H_{n-2}(X_a^1(L_{\Delta^{n+1}});{\Bbb Z}_2)\cong
\bigoplus_{i+j=n-2}\mathcal{C}_0(K_{i,j}^{n+1};{\Bbb Z}_2).$$ Since
$K_{i,j}^{n+1}$ with $i+j=n-2$ is connected, take a simplex $a\in
K_{i,j}^{n+1}\subset L_{\Delta^{n+1}}$ with $\dim a=0$ (i.e., $a$ is
a vertex of $K_{i,j}^{n+1}$ and it is also a face of $\Delta^{n+1}$ of dimension $\leq n-2$), each generator $$(\beta^i\otimes
\gamma^j)_a\in H_{n-2}(X_a^1(L_{\Delta^{n+1}});{\Bbb Z}_2)$$ is a
cycle  in the chain group $E_{0,n-2}^1(L_{\Delta^{n+1}}, 1)\otimes
{\Bbb Z}_2$ (so the whole $E_{0,n-2}^1(L_{\Delta^{n+1}}, 1)\otimes
{\Bbb Z}_2$ exactly becomes the cycle chain group in this case), and
in particular,  it is homologous to any one of other generators
$(\beta^i\otimes \gamma^j)_{a'}, a'\in K_{i,j}^{n+1}$ with $\dim
a'=0$. With this understood, for each $0\leq i\leq n-2$, we can
always choose a face
$\sigma$ of dimension $i$ in $\Delta^n\subset \Delta^{n+1}$ as such a $0$-simplex $a$ in $K_{i,j}^n \subset
K_{i,j}^{n+1}$, so that when $i$ runs over $0, 1, ..., n-2$, all $n-1$ generators of $E_{0,n-2}^2(L_{\Delta^{n}}, 1)\otimes {\Bbb Z}_2$ (resp. $E_{0,n-2}^2(L_{\Delta^{n+1}}, 1)\otimes {\Bbb Z}_2$) can be represented by $n-1$
elements of the forms $(\beta^i\otimes \gamma^j)_a\in H_{n-2}(X_a^1(L_{\Delta^{n}});{\Bbb Z}_2)$ (resp.
$H_{n-2}(X_a^1(L_{\Delta^{n+1}});{\Bbb Z}_2)$), respectively. 

Now for each such $a$,  consider the following commutative diagram
\begin{equation}\label{dia-com}
\begin{CD}
H_{n-2}(X^1_a(L_{\Delta^{n}});{\Bbb Z}_2)@>{\theta_a}>>H_{n-2}(X_a^1(L_{\Delta^{n+1}});{\Bbb Z}_2)\\
@VVV @V\lambda VV\\
H_{n-2}(X_1(M^{n});{\Bbb Z}_2)@>>>H_{n-2}(X_1(M^{n+1});{\Bbb Z}_2)
\end{CD}
\end{equation}
where the four homomorphisms are induced by  natural embeddings,
such that $\theta_a$ maps the  generator $(\beta^i\otimes \gamma^j)_a$ of
$H_{n-2}(X^1_a(L_{\Delta^{n}});{\Bbb Z}_2)$ to that of
$$H_{n-2}(X_a^1(L_{\Delta^{n+1}});{\Bbb Z}_2).$$
Next we shall show that  the following diagram is commutative
\begin{equation}\label{diagram}
\begin{CD}
H_{n-2}(X_a^1(L_{\Delta^{n+1}});{\Bbb Z}_2)@>{\kappa}>>E_{0,n-2}^1(L_{\Delta^{n+1}}, 1)\otimes {\Bbb Z}_2\\
@V\lambda VV @V{p} VV\\
H_{n-2}(X_1(M^{n+1});{\Bbb Z}_2)@<{\omega}<{\cong}<E_{0,n-2}^2(L_{\Delta^{n+1}}, 1)\otimes {\Bbb Z}_2\\
\end{CD}
\end{equation}
where  $\kappa$ is the natural inclusion, $p$ is the quotient map,
and $\omega$ is the isomorphism induced by the spectral sequence.
In fact,  from the definition of our spectral sequence in
Section~\ref{MV}, we may induce the following homomorphism
$$
E^1_{1,*}(L_{\Delta^{n+1}}, 1)\otimes {\Bbb Z}_2\xrightarrow[]{{\bf d}_1}
E^1_{0,*}(L_{\Delta^{n+1}}, 1)\otimes {\Bbb Z}_2
\xrightarrow[]{i_*}
 H_*(X_1(M^{n+1});{\Bbb Z}_2)
$$
$$\parallel$$
$$ \bigoplus_{a\in L_{\Delta^{n+1}}\\
\atop \dim a=0}H_{*}(X_a^1(L_{\Delta^{n+1}});{\Bbb
Z}_2)$$
 where $i_*$
is induced by all embeddings $i_a:
X_a^1(L_{\Delta^{n+1}})\hookrightarrow X_1(M^{n+1})$. Then  $E^2_{0,*}(L_{\Delta^{n+1}}, 1)\otimes {\Bbb Z}_2$ is the cokernel of ${\bf d}_1$, so $E^2_{0,*}(L_{\Delta^{n+1}}, 1)\otimes {\Bbb Z}_2$ is a subgroup of
$H_*(X_1(M^{n+1});{\Bbb Z}_2)$.
Thus, the quotient map $p$ agrees with $\lambda$, i.e., the above diagram (\ref{diagram}) is commutative.

Furthermore,  let $i$ ran over $0, 1, ..., n-2$. Then,  all $\lambda\theta_a(\beta^i\otimes \gamma^{n-2-i})_a, i=0,1,..., n-2$, form a basis of $H_{n-2}(X_1(M^{n+1});{\Bbb Z}_2)$. Using the commutative diagram (\ref{dia-com}), we see  that $\dim H_{n-2}(X_1(M^{n});{\Bbb Z}_2)\geq \dim H_{n-2}(X_1(M^{n+1});{\Bbb Z}_2)=n-1$.
On the other hand, we have that $\dim H_{n-2}( X_1(M^n);{\Bbb Z}_2)=\dim E_{0, n-2}^\infty(L_{\Delta^n}, 1)\otimes {\Bbb Z}_2\leq \dim E_{0, n-2}^2(L_{\Delta^n}, 1)\otimes {\Bbb Z}_2=n-1$. Therefore, we have that $\dim H_{n-2}( X_1(M^n);{\Bbb Z}_2)= n-1$, so Claim C holds. Moreover, we may obtain  that the differential ${\bf d}_r$ in (\ref{diff}) is a zero homomorphism for $r\geq 2$.
 This completes the proof.
\end{proof}

\begin{proof}[Proof of Proposition~\ref{b2}]
Using Theorem~\ref{eq-coh}, Proposition~\ref{sm-sq} and
Corollary~\ref{e2-hom}, by direct calculations we may induce the
required result for (mod 2) Betti numbers. For more detailed
calculations, we would like leave them as an exercise to the
readers.
\end{proof}

\section{Appendix--Mayer-Vietoris spectral sequence } \label{MV}

Suppose that  $X$ is a  CW-complex with all cells indexed by $J$, and $X_1,...,X_N$ are
subcomplexes of $X$ such that $\bigcup_iX_i=X$ and all possible intersections of $X_1,...,X_N$ are subcomplexes of $X$.  Associated with $X$, we may define an abstract
simplicial complex $K$ (including empty set) with vertices $1,...,N$ (or $X_1,...,X_N$) as follows: if $X_{i_1}\cap\cdots\cap X_{i_r}\not=\emptyset$, then
$\{i_1, ..., i_r\}\in K$.  For each $a\in K$,  we define
$$X_a=\begin{cases}
\bigcap_{i\in a}X_i & \text{ if } a\not=\emptyset\\
X & \text{ if } a=\emptyset.
\end{cases}$$
Set $D^K_{p,q}(X)=\bigoplus_{a\in K, |a|=p+1}D_q(X_a)$ where
$D_*(X_a)=\{D_q(X_a)\}$ is the cellular chain complex of $X_a$.
Then we shall see that $D^K_{*,*}(X)$ has a natural double complex structure.

Let $e_\alpha$ be
a cell of $X$ in $\{e_\alpha| \alpha\in J\}$. Define $K(e_\alpha)=\{a\in K\mid e_\alpha\subset
X_a\}$. Obviously, $K(e_\alpha)$ is a subcomplex  determined by some simplex of $K$, so it is acyclic.
If  $a\in K(e_\alpha)$, then  $e_\alpha$ would be  a
generator of $D_{\dim e_\alpha}(X_a)$, denoted by
$e_{\alpha,a}$. Furthermore, we may write each cellular chain of $D^K_{p,q}(X)=\bigoplus_{a\in K, |a|=p+1}D_q(X_a)$ as
$$\sum_{\alpha\in J(q)}\sum_{a\in K(e_\alpha) \atop |a|=p+1}k_{\alpha,a} e_{\alpha,a}$$ where
$k_{\alpha,a}\in {\Bbb Z}$, and $J(q)$ means that for $\alpha\in J(q)$, $\dim e_\alpha=q$.

Let $c=\sum_{a\in K(e_\alpha)}\lambda_a a$ be a chain in the simplicial chain complex $\mathcal{C}_*(K(e_\alpha))$
of
$K(e_\alpha)$.
Define $e_{\alpha, c}=\sum_{a\in K(e_\alpha)}\lambda_a e_{\alpha, a}$. Then it is easy to check that
\begin{lem}\label{su-ne}
$e_{\alpha, c}=0$ if and only if $c=0$.
\end{lem}

Now two differentials  on $D^K_{*,*}(X)$ are defined as follows:
One is $\partial_1:D^K_{p,q}(X)\rightarrow D^K_{p,q-1}(X)$  given by
$$\partial_1(\sum_{\alpha\in J(q)}\sum_{a\in K(e_\alpha) \atop |a|=p+1}k_{\alpha,a} e_{\alpha,a})=\sum_{\alpha\in J(q)}\sum_{a\in K(e_\alpha) \atop |a|=p+1}k_{\alpha,a} \partial(e_{\alpha,a})$$
which is induced by the boundary homomorphism $\partial$ of $D_*(X_a)$, and the other one is   $\partial_2:D^K_{p,q}(X)\rightarrow
D^K_{p-1,q}(X)$ given by
\begin{equation}\label{eq}
\partial_2(\sum_{\alpha\in J(q)}\sum_{a\in K(e_\alpha) \atop |a|=p+1}k_{\alpha,a} e_{\alpha,a})=\sum_{\alpha\in J(q)}\sum_{a\in K(e_\alpha) \atop |a|=p+1}k_{\alpha,a} e_{\alpha,\partial' a}
\end{equation}
which is induced by the boundary homomorphism $\partial'$ of the simplicial chain complex $\mathcal{C}_*(K(e_\alpha))$.
Note that for the empty set  $\emptyset\in K$,  $\partial'\emptyset=0$.

\vskip .2cm An easy argument shows  that
$\partial_1\partial_2=\partial_2\partial_1$. Thus we have the
following commutative diagram:

\[
\begin{CD}
@. 0 @. 0 @.  @. 0  \\
@. @A\partial_2 AA @A\partial_2 AA @. @A\partial_2 AA  \\
0@<\partial_1<<D^K_{-1,0}(X) @<\partial_1<< D^K_{-1,1}(X)  @<\partial_1<< \cdots @<\partial_1<< D^K_{-1,q}(X)@<\partial_1<<\cdots\\
@. @A\partial_2 AA @A\partial_2 AA @. @A\partial_2 AA \\
0@<\partial_1<<D^K_{0,0}(X) @<\partial_1<< D^K_{0,1}(X) @<\partial_1<< \cdots @<\partial_1<< D^K_{0,q}(X)@<\partial_1<<\cdots\\
@. @A\partial_2 AA @A\partial_2 AA @. @A\partial_2 AA  \\
\vdots @. \vdots @. \vdots @.  @. \vdots @. \vdots \\
@. @A\partial_2 AA @A\partial_2 AA @. @A\partial_2 AA  \\
0@<\partial_1<<D^K_{p,0}(X) @<\partial_1<< D^K_{p,1}(X) @<\partial_1<< \cdots @<\partial_1<< D^K_{p,q}(X)@<\partial_1<<\cdots\\
@. @A\partial_2 AA @A\partial_2 AA @. @A\partial_2 AA  \\
\vdots @. \vdots @. \vdots @.  @. \vdots @. \vdots \\
\end{CD}
\]

Now, let us look at the  structure of this double complex
$(D^K_{*,*}(X), \partial_1,\partial_2)$.

\begin{prop}
Every column of the above diagram is exact, i.e., for each $q$,
$$
0\xleftarrow[]{\partial_2}D^K_{-1,q}(X) \xleftarrow[]{\partial_2} D^K_{0,q}(X)  \xleftarrow[]{\partial_2} \cdots \xleftarrow[]{\partial_2} D^K_{p,q}(X)\xleftarrow[]{\partial_2}\cdots\\
$$
is exact.
\end{prop}

\begin{proof}
Suppose that $\sum_{\alpha\in J(q)}\sum_{a\in K(e_\alpha) \atop |a|=p+1}k_{\alpha,a} e_{\alpha,a}$ is  a cycle in $D^K_{p,q}(X)$. Then $$\partial_2(\sum_{\alpha\in J(q)}\sum_{a\in K(e_\alpha) \atop |a|=p+1}k_{\alpha,a} e_{\alpha,a})=0.$$
Furthermore, we have that for each $\alpha\in J(q)$, $\partial_2(\sum_{a\in K(e_\alpha)
\atop |a|=p+1}k_{\alpha,a} e_{\alpha,a})=0$.
By Lemma~\ref{su-ne}, we obtain that for each $\alpha\in J(q)$,
 $$\partial'(\sum_{a\in K(e_\alpha) \atop |a|=p+1}k_{\alpha,a}a)=0,$$ so $\sum_{a\in K(e_\alpha)
 \atop |a|=p+1}k_{\alpha,a}a$ is a cycle in $\mathcal{C}_p(K(e_\alpha))$. Since $K(e_\alpha)$ is acyclic,
 there exists a chain $c_\alpha$ in $\mathcal{C}_{p+1}(K(e_\alpha))$ such that
 $\partial' c_\alpha=\sum_{a\in K(e_\alpha) \atop |a|=p+1}k_{\alpha,a}a$,
so $$\sum_{a\in K(e_\alpha) \atop |a|=p+1}k_{\alpha,a}e_{\alpha, a}
=e_{\alpha, \sum_{a\in K(e_\alpha) \atop |a|=p+1}k_{\alpha,a}a}
=e_{\alpha, \partial' c_\alpha}=\partial_2 (e_{\alpha, c_\alpha}).$$
Therefore, we conclude that
$$\sum_{\alpha\in J(q)}\sum_{a\in K(e_\alpha) \atop |a|=p+1}k_{\alpha,a} e_{\alpha,a}=\partial_2 (\sum_{\alpha\in J(q)}e_{\alpha, c_\alpha})$$
is also a boundary chain, as desired.
\end{proof}
\vskip .2cm

Define $E^1_{p,q}(K)$ as the $q$-th homology group of the $p$-th row in the above diagram for $p\geq 0$, and when $p<0$, let  $E^1_{p,q}(K)=0$.
Let $\IM_{p,q}^K\partial_2=\IM(D^K_{p,q}(X)\xrightarrow[]{\partial_2}
D^K_{p-1,q}(X))$. Then we have the induced chain complex:
\[
\begin{CD}
0\longleftarrow\IM_{p,0}^K\partial_2 \longleftarrow \IM_{p,1}^K\partial_2\longleftarrow \IM_{p,2}^K\partial_2 \longleftarrow \cdots
\longleftarrow \IM_{p,q}^K\partial_2\longleftarrow\cdots
\end{CD}
\]
Let $A^1_{p,q}(K)$ be the $q$-th homology group of this chain complex. Note that when
$p<0$,  let $A^1_{p,q}(K)=A^1_{0, p+q}(K)$.

\vskip .2cm
Because every column in the above diagram is exact,  we have the following short exact sequence:
\[
\begin{CD}
0 @>>> \IM_{p+1,*}^K\partial_2@>>> D^K_{p,*}(X) @>\partial_2>>
\IM_{p,*}^K\partial_2 @>>> 0
\end{CD}
\]
Furthermore, we may obtain the following long exact sequence:
$$
\cdots\xrightarrow[]{i} A^1_{p+1,q}(K) \xrightarrow[]{j} E^1_{p,q}(K) \xrightarrow[]{k} A^1_{p,q}(K)
\xrightarrow[]{i} A^1_{p+1,q-1}(K) \xrightarrow[]{j}
E^1_{p,q-1}(K)\xrightarrow[]{k}\cdots
$$
Then we can induce the following exact couple
$$\xymatrix{
A(K) \ar[rr]^i&&A(K) \ar[dl]^j\\
  &E(K)\ar[lu]^k& }$$ and the  spectral sequence $E^1_{p,q}(K), E^2_{p,q}(K), ..., E^\infty_{p,q}(K)$ such that
  the $r$-th differential $\text{\bf d}_r$ on $E^r_{p,q}(K)$ is of bidegree $(-r, r-1)$.

  \begin{rem}\label{complex}
We can explicitly write out the  differential $\text{\bf d}_1=j\circ k$
 of the  chain complex
$$
 \cdots\xrightarrow[]{j\circ k} E^1_{p,q}(K) \xrightarrow[]{j\circ k} E^1_{p-1,q}(K) \xrightarrow[]{j\circ k} \cdots\xrightarrow[]{j\circ k} E^1_{0,q}(K)
\longrightarrow0.
$$
At first, for $a\in K$, given an element $\beta_a\in H_*(X_a)$. Let
$b\subset a$. Then we define $\beta_{a,b}$ as the image of $\beta_a$
under the map $H_*(X_a)\rightarrow H_*(X_b)$ induced by the natural
imbedding $X_a\hookrightarrow X_b$. Now if $c=\sum_i\lambda_ib_i$ is
a chain of the simplicial chain complex of $K$ where $b_i\subset a$,
define $\beta_{a,c}$ as $\sum_i\lambda_i\;\beta_{a,b_i}$. Next, by
the definition of $E(K)$, we know that
\begin{equation}\label{e1}
E^1_{p,q}(K)\cong \bigoplus_{a\in K\atop |a|=p+1}H_q(X_a)
\end{equation}
so we can write its element as $\sum_{a\in K\atop |a|=p+1}\beta_a$.
Moreover, we see easily that
$$j\circ k(\sum_{a\in K\atop
|a|=p+1} \beta_a)=\sum_{a\in K\atop
|a|=p+1} \beta_{a,\partial'
a}.$$
In particular, when $q=0$ and every $X_a$ is connected, the chain
complex
$$
\cdots\xrightarrow[]{j\circ k} E^1_{p,0}(K) \xrightarrow[]{j\circ k} E^1_{p-1,0}(K) \xrightarrow[]{j\circ k}\cdots \xrightarrow[]{j\circ k} E^1_{0,0}(K)\longrightarrow0
$$
is isomorphic to the simplicial chain complex of $K\setminus
\{\emptyset\}$. Thus, $E^2_{p,0}(K)$ is isomorphic to $H_p(K\setminus
\{\emptyset\})$. Note that
since we have assumed that $\emptyset\in K$, $H_p(K)$ is isomorphic to the reduced homology $\widetilde{H}_k(K\setminus
\{\emptyset\})$.
\end{rem}

By the theory of spectral sequence, we have that
\begin{thm}\label{spectral}
$$H_i(X)\cong \bigoplus_{p+q=i}E^\infty_{p,q}(K)$$
\end{thm}

Generally, $K$ may have a very complicated structure. This will lead to a difficulty for calculating
  the spectral sequence $E^1_{p,q}(K),E^2_{p,q}(K),...,E^\infty_{p,q}(K)$ induced by the double complex $D^K_{*,*}(X)$. For the purpose of our application, we shall choose a suitable subcomplex of $K$, so that this may give a simpler
  calculation.

\begin{defn}\label{local}
A subcomplex $L$ of $K$ is said to be {\em locally nice} if $L$ satisfies the following properties:
\begin{enumerate}
\item[$\bullet$] $L$ contains all vertices of $K$ and the empty set $\emptyset$.
\item[$\bullet$] For each cell $e_\alpha$ of $X$, $L(e_\alpha)=\{a\in L| e_\alpha\subset X_a\}$ is acyclic.
\end{enumerate}
\end{defn}

Now let $L$ be a locally nice subcomplex of $K$. Similarly, we can define a double complex
$D_{*,*}^L(X)=\{D_{p,q}^L(X)\}$, where $D^L_{p,q}(X)=\bigoplus_{a\in L, |a|=p+1}D_q(X_a)$. Then we see that
for each $q$,
\[
\begin{CD}
0@<\partial_2<<D^L_{-1,q}(X) @<\partial_2<< D^L_{0,q}(X)  @<\partial_2<< \cdots @<\partial_2<< D^L_{p,q}(X)@<\partial_2<<\cdots\\
\end{CD}
\]
is still exact since $L(e_\alpha)=\{a\in L| e_\alpha\subset X_a\}$ is acyclic for each cell $e_\alpha$ of $X$.
Thus, we can induce a corresponding spectral sequence $E^1_{p,q}(L),E^2_{p,q}(L),...,E^\infty_{p,q}(L)$
such that
$$H_i(X)\cong\bigoplus_{p+q=i}E^\infty_{p,q}(L).$$

\vskip .2cm

It should be pointed out that there is also a cohomological version of the above argument. Namely, we can obtain a spectral sequence $E_1^{p,q}(L),E_2^{p,q}(L),...,E_\infty^{p,q}(L)$ from $(X, L)$ such that
$$H^i(X)\cong \bigoplus_{p+q=i}E_\infty^{p,q}(L).$$

\end{document}